\documentclass[11pt,twoside,leqno,a4paper,final]{amsart}
\usepackage[a4paper=true,pdfpagelabels,colorlinks]{hyperref} 
\usepackage{amsmath}
\usepackage{mathtools,nccmath}

\usepackage{amssymb}
\usepackage{enumitem}

\usepackage[dvipsnames]{xcolor}

\usepackage{mathrsfs} 

\newcommand{\vertiii}[1]{{\left\vert\kern-0.25ex\left\vert\kern-0.25ex\left\vert #1 
    \right\vert\kern-0.25ex\right\vert\kern-0.25ex\right\vert}}

\theoremstyle{plain}
\newtheorem{proposition}{Proposition}[section]
\newtheorem{corollary}[proposition]{Corollary}
\newtheorem{lemma}[proposition]{Lemma}
\newtheorem{theorem}[proposition]{Theorem}

\theoremstyle{definition}
\newtheorem{definition}[proposition]{Definition}

\newtheorem{remark}[proposition]{Remark}

\newcommand{\A}{\mathcal{A}}

\newcommand{\B}{\mathscr{B}}
\newcommand{\BB}{\mathcal{B}}

\newcommand{\C}{{\mathbb C}}
\newcommand{\CC}{\mathcal{C}}

\newcommand{\D}{{\mathbb D}}
\newcommand{\DD}{\mathcal{D}}

\renewcommand{\H}{\mathcal{H}}

\newcommand{\N}{{\mathbb N}}

\newcommand{\T}{\mathbb{T}}

\renewcommand{\DD}{\widehat{\mathcal{D}}}
\newcommand{\Dd}{\check{\mathcal{D}}}
\newcommand{\M}{\mathcal{M}}
\newcommand{\DDD}{\mathcal{D}}

\def\om{\omega}
\def\W{{\mathcal W}}
\def\TT{{\mathcal T}}

\renewcommand{\M}{\mathcal{M}}
\newcommand{\Mult}{\operatorname{Mult}}

\newcommand{\spn}{\operatorname{span}}

\usepackage{bm}

\makeatletter
\newcommand{\opnorm}{\@ifstar\@opnorms\@opnorm}
\newcommand{\@opnorms}[1]{%
  \left|\mkern-1.5mu\left|\mkern-1.5mu\left|
   #1
  \right|\mkern-1.5mu\right|\mkern-1.5mu\right|
}
\newcommand{\@opnorm}[2][]{%
  \mathopen{#1|\mkern-1.5mu#1|\mkern-1.5mu#1|}
  #2
  \mathclose{#1|\mkern-1.5mu#1|\mkern-1.5mu#1|}
}
\makeatother

\makeatletter
\newcommand{\boldopnorm}{\@ifstar\@boldopnorms\@boldopnorm}
\newcommand{\@boldopnorms}[1]{%
  \pmb{\left|}\mkern-1.5mu\pmb{\left|}\mkern-1.5mu\pmb{\left|}
   #1
  \pmb{\right|}\mkern-1.5mu\pmb{\right|}\mkern-1.5mu\pmb{\right|}
}
\newcommand{\@boldopnorm}[2][]{%
  \mathopen{#1\pmb{|}\mkern-1.5mu#1\pmb{|}\mkern-1.5mu#1\pmb{|}}
  #2
  \mathclose{#1\pmb{|}\mkern-1.5mu#1\pmb{|}\mkern-1.5mu#1\pmb{|}}
}
\makeatother

\numberwithin{equation}{section}

\newtheorem{lettertheorem}{Theorem}

\theoremstyle{plain} 
\newcommand{\thistheoremname}{}
\newtheorem*{genericthm*}{\thistheoremname}
\newenvironment{namedthm*}[1]
{\renewcommand{\thistheoremname}{#1}%
	\begin{genericthm*}}
	{\end{genericthm*}}

\theoremstyle{definition} 
\newcommand{\thisdefinitionname}{}
\newtheorem*{genericdefinition*}{\thisdefinitionname}
\newenvironment{nameddefinition*}[1]
{\renewcommand{\thisdefinitionname}{#1}%
	\begin{genericdefinition*}}
	{\end{genericdefinition*}}

\makeatletter
\@namedef{subjclassname@2020}{\textup{2020} Mathematics Subject Classification}
\makeatother

\begin{document}

\title
[Radicality property for spaces  of symbols]
{ On the radicality property for spaces  of symbols of  bounded Volterra  operators}

\author[C. Cascante]{Carme Cascante}
\address{C. Cascante: Departament de Matem\`atiques i
    Inform\`atica, Universitat  de Barcelona,
     Gran Via 585, 08071 Barcelona, Spain \&
Centre de Recerca Matem\`atica, Edifici C, Campus Bellaterra, 08193 Bellaterra, Spain}
\email{cascante@ub.edu}
\author[J. F\`abrega]{Joan F\`abrega}
\address{J. F\`abrega: Departament de Matem\`atiques i
    Inform\`atica, Universitat  de Barcelona,
     Gran Via 585, 08071 Barcelona, Spain}
\email{joan$_{-}$fabrega@ub.edu}
\author[D. Pascuas]{Daniel Pascuas}
\address{D. Pascuas: Departament de Matem\`atiques i
    Inform\`atica, Universitat  de Barcelona,
     Gran Via 585, 08071 Barcelona, Spain}
\email{daniel$_{-}$pascuas@ub.edu}
\author[J. A. Pel\'aez]{Jos\'e \'Angel Pel\'aez}
\address{J. A. Pel\'aez: Departamento de An\'alisis Matem\'atico, Universidad de M\'alaga, Campus de
Teatinos, 29071 M\'alaga, Spain} \email{japelaez@uma.es}

\thanks{
The research of the first, second, and third authors
	was supported in part by
        Ministerio de Ciencia e Innovaci\'{o}n, 
        Spain, project  PID2021-123405NB-I00.  The first author is also supported by the Spanish State Research Agency, through the Severo Ochoa and Mar\'{\i}a de Maeztu Program for Centers and Units of Excellence in R\&D 
(CEX2020-001084-M), and by Departament de Recerca i Universitats  grant SGR 2021-00087.
The research of the fourth author was supported in part by Ministerio de Ciencia e Innovaci\'on, Spain, project  PID2022-136619NB-I00; La Junta de Andaluc{\'i}a,
project FQM210.}
\date{\today}

\subjclass[2020]{30H10, 30H20, 47G10, 30H35, 30H30}

\keywords{Volterra-type operator, Analytic paraproduct, iterated composition of operators, 
 weighted Bergman spaces, Bloch space}

\date{\today}

\begin{abstract} 
In \cite{Aleman:Cascante:Fabrega:Pascuas:Pelaez} it is shown that the Bloch space $\mathcal{B}$ in the unit disc has the following 
radicality property: if an analytic function $g$ satisfies that $g^n\in \mathcal{B}$, then $g^m\in \mathcal{B}$, for all $m\le n$. Since $\mathcal{B}$ coincides with the space  $\mathcal{T}(A^p_\alpha)$  of analytic symbols $g$  such that the Volterra-type operator  
$T_gf(z)= \int_0^z f(\zeta)g'(\zeta)\,d\zeta$
 is bounded on the classical weighted Bergman space $A^p_\alpha$, the radicality property was used to study the composition of paraproducts $T_g$ and $S_gf=T_fg$ on $A^p_{\alpha}$. Motivated by this fact, we prove that $\mathcal{T}(A^p_\omega)$ also has the radicality property, for any radial weight $\omega$. Unlike the classical case, 
the lack of a precise description of $\mathcal{T}(A^p_\omega)$ for a general radial weight, induces us to prove the radicality property for $A^p_\omega$ from precise norm-operator results for compositions of analytic paraproducts.

\end{abstract}
\maketitle

\noindent
\section{Introduction} 

Let $\H(\D)$ denote the algebra of all analytic functions in the unit disc $\D$ of the complex plane $\mathbb{C}$. A function $\omega:\D\to [0,\infty)$, integrable over $\D$, is called a {\em weight}. For $0<p<\infty$ and a weight $\omega$, the weighted Bergman
space $A^p_\omega$ consists of those $f\in\H(\D)$ for which
    $$
    \|f\|_{A^p_\omega}^p=\int_\D|f(z)|^p\omega(z)\,dA(z)<\infty,\index{$\Vert\cdot\Vert_{A^p_\omega}$}
    $$
where $dA(z)=\frac{dx\,dy}{\pi}$ is the normalized Lebesgue area measure on $\D$.
A weight is {\em radial} if $\omega(z)=\omega(|z|)$, for all $z\in\D$,  $\int_0^1\om(s)\,ds<\infty$ and 
$\widehat{\omega}(r)=\int_r^1 \omega(s)\,ds>0$, for any $r\in [0,1)$. If the last hypothesis does not hold $A^p_\omega=\H(\D)$.
 As usual,  we write $A^p_\alpha$ for the Bergman space induced by the  standard weight 
$\omega(z)=(\alpha+1)(1-|z|^2)^\alpha$, $\alpha>-1$.  
Throughout the manuscript  the space of bounded  linear operators on $A^p_\omega$ is denoted by $\BB(A_\omega^p)$, and for  any linear map  $L:\H(\D)\to\H(\D)$ we write
$\|L\|_{A_\omega^p}:=\sup \{\|Lf\|_{A_\omega^p}:\|f\|_{A_\omega^p}=1\}$. We refer to this quantity as the operator norm of $L$ on $A^p_\omega$, despite $A^p_\omega$ is not a normed space for $0 <p <1$.

For any $g\in\H(\D)$, we consider the {\em Volterra-type operator} 
\[
T_gf(z):= \int_0^z f(\zeta)g'(\zeta)\,d\zeta \qquad(f\in\H(\D),\,z\in\D).
\]  
In this paper we are interested in the spaces of analytic functions 
\[
\TT(A^p_\omega):=\{ g\in \H(\D): T_g \in  \BB(A_\omega^p) \}\mbox{ with the seminorm } \|g\|_{\TT(A^p_\omega)}:= \|T_g\|_{A_\omega^p}.
\] 
It is well-known   that $\TT(A^p_\alpha)=\B$, the Bloch space, and recently
the  conformally invariance of the Garsia's seminorm $\vertiii{g}_{\B}:=\sup_{a\in\D}\|g\circ\phi_a-g(a)\|_{A^2}$,  $\phi_a(z):=\frac{a-z}{1-\overline{a}z}$, has been strongly used to prove   the following meaningful property of the Bloch space \cite[Section 2]{Aleman:Cascante:Fabrega:Pascuas:Pelaez}.
\begin{lettertheorem}\label{th:powerBloch}
Let $m,n\in\N$, $m<n$, and $g\in\H(\D)$. If $g^n\in\B$, then $g^m\in\B$ and 
	\begin{equation*}
	\vertiii{g^m}_{\B}^{1/m}\le\vertiii{g^n}_{\B}^{1/n}.
	\end{equation*}
\end{lettertheorem}
It is worth noticing that Theorem~\ref{th:powerBloch} is a pivotal result within the  theory of composition of analytic paraproducts on classical Bergman and Hardy spaces \cite{Aleman:Cascante:Fabrega:Pascuas:Pelaez,Aleman:Cascante:Fabrega:Pascuas:Pelaez2}. Let us recall the reader that for any $g\in\H(\D)$,  besides $T_g$, the operators  
\begin{equation*}
M_gf:= fg
\quad\mbox{ and }\quad 
S_gf(z):= \int_0^z f'(\zeta)g(\zeta)\,d\zeta
\end{equation*} are called  
 {\em $g$-analytic paraproducts}.

Theorem~\ref{th:powerBloch}  leads us to introduce the following concept: A space $X$ of analytic functions in $\D$  has the  {\em radicality property}
if for any $g\in \H(\D)$ and $n\in\N$ such that $g^n\in X$, then $g^m \in X$ for all $m\in \N$ such that $m<n$. This definition is inspired by the ideal theory in Commutative Algebra. 
Consequently,   $\TT(A^p_\alpha)=\B$ satisfies the radicality property  and  the next natural question arises:

\begin{quote}
 {\em Given $0<p<\infty$, which are the weights such that $\TT(A^p_\omega)$ has the radicality property?}
\end{quote}

Of course, by Theorem \ref{th:powerBloch} the answer is obvious for any radial weight $\omega$ for which $\TT(A^p_\omega)=\B$. We remark that, besides standard weights, Bekoll\'e-Bonami weights and 
radial doubling weights \cite{AlCo,PelSum14} satisfy that  $\TT(A^p_\omega)=\B$, for any $p\in (0,\infty)$.
 In general, the situation is much more difficult because the  existing literature does not provide a description of    $\TT(A^p_\omega)$, even in the case when $\omega$ is radial. In addition,  it is worth recalling the existence of classes of weights $\omega$ such that a handy description   of    $\TT(A^p_\omega)$ is known but
 $\TT(A^p_\omega)$ is not conformally invariant, so despite having a characterization of  $\TT(A^p_\omega)$
tackling the question above may require different techniques to those employed in the proof of Theorem~\ref{th:powerBloch}.  For instance,  this happens  if $\omega$ belongs to
the class of rapidly decreasing weights $\W$  and may happen if $\omega$ belongs to the class $\DD$
of all radial weights $v$ such that $\sup_{0\le r<1}\frac{\widehat{v}(r)}{\widehat{v}(\frac{1+r}{2})}<\infty$. Operator theory on weighted Bergman spaces $A^p_\omega$ induced by weights 
in $\DD$ or $\W$  has attracted a lot of attention in the last decade,
see  Section \ref{S:radialweights} below  for the definition of the class $\W$ and further details about a description of  $\TT(A^p_\omega)$ when $\om\in\W$ or $\om\in\DD$.
Our main result is the following.

\begin{theorem}\label{cor:compoq} Let    $\omega$ be a radial weight and $0<p<\infty$. Then  $\TT(A^p_\omega)$ satisfies the radicality property.  Moreover,
if $m,n\in\N$, $m<n$, then 
\[
\|g^m\|^{\frac1m}_{\TT(A^p_\omega)}\lesssim\|g^n\|^{\frac1n}_{\TT(A^p_\omega)}
\qquad(g\in\H(\D)).
\]
\end{theorem}
As usual,  $A\lesssim B$ ($B\gtrsim A$) for nonnegative functions $A$, $B$ means that $A\le C\,B$, for some positive constant $C$
 independent of the variables involved. Furthermore, we write $A\simeq B$ when $A\lesssim B$ and $A\gtrsim B$.

Before providing some words on the proof of Theorem~\ref{cor:compoq}, we point out that there are spaces $X$ so that the radicality property does not hold for $\TT(X)$. Indeed, for $0<s<1$ let us consider the space of $s$-H\"older analytic functions  $Lip_s=\{f\in\H(\D):\sup_{z\in\D}(1-|z|)^{1-s}|f'(z)|<\infty\}$.  Bearing in mind that $Lip_s\subset H^{\infty}$, it is not difficult to prove that $\TT(Lip_s)=Lip_s$. Therefore $\TT(Lip_s)$ does not satisfy the radicality property because the function $g(z)=(1-z)^{s/2}$ does not belong to $Lip_s$, while $g^2(z)=(1-z)^s$ does.

The proof of Theorem~\ref{cor:compoq} is strongly based on the theory on composition of analytic paraproducts. Therefore  
 to give a brief explanation of its proof we will  remind some basic definitions of that theory and state some  results which are of interest in themselves.
A {\em $g$-word} is a composition (product) of $g$-analytic paraproducts. Namely, an {\em $N$-letter $g$-word} is an operator of the form $L=L_1\cdots L_N$, where each $L_j$ is either $M_g$, $S_g$ or $T_g$.  
By convention, the identity mapping $I$ on $\H(\D)$ is the only $0$-letter $g$-word. 
Moreover, a {\em $g$-operator} is a  linear combination of $g$-words, which may have different number of letters. The algebra $\A_g$ is  the set of all $g$-operators.

The formula $L_g=L_g\Pi_0+(L_g1)\delta_0$, where $\Pi_0f:=f-f(0)$, together with the $ST$-representation of each $g$-operator $L_g$ proved in~{\cite[\S 3.2]{Aleman:Cascante:Fabrega:Pascuas:Pelaez}} gives that
\begin{equation}\label{eq:STrepre}
	L_g=\sum_{k=0}^NS_g^kT_gP_k(T_g)\Pi_0+S_gP_{N+1}(S_g)+
	P_{N+2}(g-g(0),g(0))\,\delta_0,
\end{equation}
where $N\in\N_0:=\N\cup\{0\}$, all the $P_k$'s, $k=0,\cdots,N+1$ are polynomials of one variable and $P_{N+2} $ is a polynomial of two variables such that $L_g1=P_{N+2}(g-g(0),g(0))$.

When $P_k=0$, for $k=0,\dots,N+1$, we will say that $L_g$ is a {\em trivial $g$-operator}. 
The norm of these one-rank operators are given by  
\[
\|L_g\|_{A^{p}_\omega}=\|P_{N+2}(g-g(0),g(0))\|_{A^{p}_\omega}\,\|\delta_0\|_{A^{p}_\omega}.
\]

 From now on, we will use the following notations:
\[
\H_0(\D):=\{f\in\H(\D): f(0)=0\}
\quad\mbox{and}\quad A_\omega^p(0):= A_\omega^p\cap \H_0(\D).
\]

\begin{theorem}\label{thm:compo}
Let $\omega$ be a radial weight, $g\in \H(\D)$,  and $0<p<\infty$.
If a non-trivial $g$-operator $L_g$ is bounded on $A_\omega^p(0)$, then $T_g$ is bounded on $A^p_\omega$.
\end{theorem}

We point out that the technical hypothesis ``$L_g$ is bounded on $A_\omega^p(0)$'' in the statement of Theorem~\ref{thm:compo} instead of the natural hypothesis ``$L_g$ is bounded on $A_\omega^p$'' allows us to simplify a good number of proofs throughout the paper.

 Theorem~\ref{thm:compo}  is known for standard weights \cite{Aleman:Cascante:Fabrega:Pascuas:Pelaez}, and it is
 a crucial result
to get a description of the symbols $g$ such that $L_g\in \BB(A^p_\alpha)$ for a large subclass of operators $L_g \in\A_g$   
\cite{Aleman:Cascante:Fabrega:Pascuas:Pelaez,Aleman:Cascante:Fabrega:Pascuas:Pelaez2}.
As for the proof of Theorem~\ref{thm:compo}, the Littlewood-Paley formula $\| f\|_{A^p_\alpha}\simeq |f(0)|+  \| f'\|_{A^p_{\alpha+p}}$ may be employed  when
$\om$ is a standard weight, however for each $p\ne2$
 there are radial weights $\omega$ (indeed $\om\in\DD$)  such that a
	Littlewood-Paley formula of type
	\begin{equation}\label{eq:NOL-P}
		\|f\|^p_{A^p_{\om}}\simeq
		|f(0)|^p+\int_\D|f'(z)|^p\varphi(|z|)^p\om(z)\,dA(z),
		\qquad(f\in\H(\D))
	\end{equation}
	is not valid for any radial function $\varphi$,  see \cite[Proposition~4.3]{Pelaez:Rattya:Memoirs} or \cite[Proposition~3.7]{PelSum14}.
	Consequently, it will be useful for our purposes to deal with a Calder\'on type formula which involves analytic tent spaces 
  and gives an equivalent norm
	to $\|f\|_{A^p_\omega}$ defined in terms of $f'$ (see Proposition~\ref{prop:Calderon:formula:tent:spaces} below for further details).
 On the other hand, an application of Theorem~\ref{thm:compo}
to the operators $L_g=S_g^{n-1}T_g=\frac{1}{n}T_{g^n}$, $n\in\N$, gives that $g\in \TT(A^p_\om)$ whenever $g^n \in \TT(A^p_\om)$.
Aiming to complete a proof of Theorem~\ref{cor:compoq} for $m\ge2$, we focus our attention in the following classes of $g$-operators. Firstly,  
we consider the  $g$-operators $L_g$ such that
\begin{equation}
	\label{eq:opLL}
	L_g=S^m_g T^n_g+\sum_{j=1}^{m} S_g^{m-j}T_g P_j(T_g)\quad\mbox{on $\H_0(\D)$},
\end{equation}
where $m,n\in\N_0$, $m+n\ge1$ and each $P_j$ is a polynomial. Here and on the following,  the sum in the right equals zero if $m=0$.  Secondly, we consider the class 
of $g$-operators $L_g$ such that
\begin{equation}\label{eq:opLLL}
	L_g=S_g^m T_g^n+\sum_{j=1}^{m} c_j S_g^{m-j}T_g^{n+j}\quad\mbox{on $\H_0(\D)$}.
\end{equation}
Note that this is a subclass of the $g$-operators $L_g$ satisfying \eqref{eq:opLL}. Moreover, for notational purposes,
if $\ell,m,n\in\N_0$ and $N=\ell+m+n\ge 1$, we define the set  $W_g(\ell,m,n)$  of $N$-letter $g$-words $L$ of the form
 \begin{equation*}
	L=L_1\cdots L_N,
\end{equation*} 
with $\#\{j:L_j=M_g\}=\ell$, $\#\{j:L_j=S_g\}=m$, and 
$\#\{j:L_j=T_g\}=n$. 
For simplicity, we write $W_g(0,m,n)=W_g(m,n)$.
We recall that \eqref{eq:opLLL} holds for any 
$L_g\in W_g(\ell,m,n)$ replacing $m$ by $m+\ell$ (see \S\ref{subsection:algebraic:results} below).

\begin{theorem}\label{thm:compoq}
	Let $\omega$ be a radial weight, $0<p<\infty$, $g\in\H(\D)$, and let $L_g$ be a $g$-operator.
	\begin{enumerate}[label={\sf\alph*)}, topsep=3pt, leftmargin=*,itemsep=3pt] 
		\item \label{item:compoq1} If $L_g$ satisfies \eqref{eq:opLL}, then  $\|T_g\|_{A_\omega^p}\lesssim \|L_g\|_{A_\omega^p(0)}^{1/(m+n)}$.
		\item \label{item:compoq2}  If $L_g$ satisfies  \eqref{eq:opLLL} and $n=0$, then $\|L_g\|_{A^p_\omega}\simeq \|S_g\|^m_{A^p_\omega}\simeq \|g\|^m_{\infty}$.   
		\item \label{item:compoq3} Assume that $L_g$ satisfies  \eqref{eq:opLLL} and $n\geq 1$. If $k\in \N_0$ and 
        $k\le\frac{m}{n}$,  then
		$S^k_gT_g \in \BB(A^p_{\omega})$ and $\| S^k_gT_g \|_{A^p_{\omega}}\lesssim \|L_g\|_{A^p_{\omega}(0)}^{\frac{k+1}{m+n}}$.
	\end{enumerate}
\end{theorem}
A more general result than Theorem~\ref{thm:compoq}, which in particular characterizes the boundedness of $N$-letter $g$-words for any $N\in\N$,   is proved for standard weights in \cite{Aleman:Cascante:Fabrega:Pascuas:Pelaez2}.
There, it is used a good number of results of  the developed operator and function theory on standard weighted Bergman  spaces, which are unknown for Bergman spaces $A^p_\om$ induced by general radial weights $\omega$. Consequently,
we are forced to employ new ideas in the proof of Theorem~\ref{thm:compoq}. Among them, we point out a handy representation of operators of the form \eqref{eq:opLLL} (see \S\ref{subsection:algebraic:results} below).

Now, observe that
 applying Theorem \ref{thm:compoq} \ref{item:compoq3} to $L_g=S_g^{n-1}T_g=\frac{1}{n}T_{g^n}$, where $n\in\N$,  we obtain   Theorem~\ref{cor:compoq}. That is, 
if  $L_g=S_g^{n-1}T_g=\frac{1}{n}T_{g^n}\in\BB(A^p_{\omega})$, then, for any $m\in\N$, $m<n$, 
$S^{m-1}_gT_g=\frac{1}{m}T_{g^{m}}\in\BB(A^p_{\omega})$
and  $\|T_{g^{m}}\|_{A^p_{\omega}}^{\frac{1}{m}}\lesssim \|T_{g^n}\|_{A^p_{\omega}}^{\frac{1}{n}}$.

 Finally, 
as a byproduct of Theorem \ref{thm:compoq}, we characterize the boundedness of the $g$-operators  $L_g$ satisfying \eqref{eq:opLLL} when $n$ divides $m$.
\begin{theorem}\label{thm:compoq:bis}
	Let $\omega$ be a radial weight and let $L_g$ be a $g$-operator.
\begin{enumerate}[label={\sf\alph*)}, topsep=3pt, leftmargin=*,itemsep=3pt] 
\item \label{item:compoq:bis1} If $L_g$ satisfies  \eqref{eq:opLLL} and $n$ divides $m$, then $\|L_g\|_{A_\omega^p}\simeq \|S^{\frac{m}n}_gT_g\|_{A_\omega^p}^n$.
\item \label{item:compoq:bis2} If $L_g\in W_g(\ell,m,n)$  and $n\ge1$ divides $\ell+m$, then $\|L_g\|_{A_\omega^p}\simeq \|S^{\frac{\ell+m}n}_gT_g\|_{A_\omega^p}^n$.
\end{enumerate}
\end{theorem}

The paper is organized as follows. 
In Section \ref{sect:preliminary:results} we prove some preliminary results  to obtain our main results. Namely, we state some decomposition formulas for $g$-operators, a basic operator approximation result, and a Calder\'{o}n type formula for $A^p_{\omega}$. Section \ref{sect:proof:thm11} is devoted to the proof of Theorem \ref{thm:compo}. 
Section \ref{sect:proof:thm13-15:prop16} deals with the proofs of 
Theorems \ref{thm:compoq} and \ref{thm:compoq:bis}. In Section \ref{sect:boundedness:single:analytic:paraproducts} we give a characterization of the boundedness of single analytic paraproducts, which for the case of $T_g$ is described in terms of pointwise multipliers.
 As a consequence of the previous results, we also obtain an embedding result for spaces of pointwise multipliers.

In the last section, we particularize our results for $A^p_{\omega}$ when either $\omega$ is a radial doubling weight or $\omega$ is a radial rapidly decreasing weight.

\section{Preliminary results}\label{sect:preliminary:results}

\subsection{Algebraic results for operators in the algebra $\A_g$}
\label{subsection:algebraic:results}

We begin this section recalling some algebraic results in $\A_g$ (see \cite{Aleman:Cascante:Fabrega:Pascuas:Pelaez} and \cite{Aleman:Cascante:Fabrega:Pascuas:Pelaez2}) from which we  obtain new algebraic formulas which will be used in the proof of Theorem \ref{thm:compoq}.

 Let $L_g\in W(\ell,m,n)$.  
By using the identities $M_g=T_g+S_g$ on $\H_0(\D)$, we can replace the operators $M_g$ in the expression of $L_g$ by $T_g+S_g$ to obtain that $L_g=\sum_{j=0}^\ell c_j Q_j$ on $\H_0(\D)$, where  
$Q_j\in W_g(m+\ell-j,n+j)$ and $c_0=1$. Next,  using the identity $T_gS_g=S_gT_g-T_g^2$  on $\H_0(\D)$, we can reorder the operators $S_g$ and $T_g$ to obtain that
\begin{equation}\label{eq:opL0}
	L_g=S_g^{\ell+m}T_g^n+\sum_{j=1}^{\ell+m} c_j S_g^{\ell+m-j}T_g^{n+j} \quad\mbox{on $\H_0(\D)$},
\end{equation}
where the $c_j$'s are complex numbers
(see  \cite[Theorem 3.1]{Aleman:Cascante:Fabrega:Pascuas:Pelaez2} for the details of the proof). In particular, any $L_g\in W_g(\ell,m,n)$ satisfies \eqref{eq:opL0}. Observe that the set of all operators satisfying \eqref{eq:opL0} coincides with the set of all operators which satisfy \eqref{eq:opLLL}

Using this fact, we are going to show that  we may replace $S^{m-j} T_g^{n+j}$ by any 
$L_j\in W_g(m-j,n+j)$ in \eqref{eq:opLLL}. Indeed, we prove a more general algebraic result, which is not included in \cite{Aleman:Cascante:Fabrega:Pascuas:Pelaez} nor in \cite{Aleman:Cascante:Fabrega:Pascuas:Pelaez2},  that will be useful to prove Theorem \ref{thm:compoq} \ref{item:compoq3}.

\begin{proposition}\label{pr:algebra1}
Let ${\mathcal L}_j\in W_g(m-j,n+j)$, $j=0,\cdots,m$, and let $L_g$ be a $g$-operator satisfying 
\begin{equation*}
L_g=L_0+\sum_{j=1}^mL_j\quad\mbox{on $\H_0(\D)$},
\end{equation*}
where $L_0\in W_g(m,n)$ and $L_j\in\spn W_g(m-j,n+j)$, for $j=1,\dots,m$.
Then  $L_g={\mathcal L}_0+\sum_{i=1}^m a_j{\mathcal L}_j$ on $\H_0(\D)$, where the $a_j$'s are complex numbers, which do not depend on $g$.
\end{proposition}
\begin{proof}
We proceed by complete induction on $m$. For $m=0$, $L_g=L_0=T_g^n={\mathcal L}_0$, and there is nothing to prove. Assume that $m>0$. Since $L_0,{\mathcal L}_0\in W_g(m,n)$, both $L_0$ and ${\mathcal L}_0$ satisfy \eqref{eq:opLLL}, so $L_0={\mathcal L}_0+\sum_{j=1}^{m} b_jS_g^{m-j}T_g^{n+j}$ on $\H_0(\D)$, where $b_j\in\C$, and therefore  we have that
\begin{equation*}
L_g={\mathcal L}_0+\sum_{j=1}^m \widetilde{L}_j
\quad\mbox{on $\H_0(\D)$},
\end{equation*}
where $\widetilde{L}_j\in\spn W_g(m-j,n+j)$, for $j=1,\dots,m$.
Now, for $j=1,\dots,m$, $\widetilde{L}_j$ is a linear combination of $g$-words in $W_g(m-j,n+j)$, and so, taking into account that ${\mathcal L}_{j+k}\in W_g((m-j)-k,(n+j)+k)$,  for $k=0,\dots, m-j$, we may apply the induction hypothesis to any of those $g$-words and get that
\[
\widetilde{L}_j=\sum_{k=0}^{m-j}a_{j,k}{\mathcal L}_{j+k},
\mbox{ on $\H_0(\D)$,}
\]
where $a_{j,k}\in\C$. Then it is clear that 
$L_g={\mathcal L}_0+\sum_{j=1}^m a_j{\mathcal L}_j$ on $H_0(\D)$, where $a_j\in\C$, and that ends the proof.
\end{proof}

A particular choice of operators  ${\mathcal L}_j\in W_g(m-j,n+j)$ in Proposition~\ref{pr:algebra1} will be particularly useful for our purpose. Namely,
\begin{corollary}\label{cor:repreT}
	Let $L_g$ be a $g$-operator satisfying 
	\begin{equation*}
		L_g=L_0+\sum_{j=1}^mL_j\quad\mbox{on $\H_0(\D)$},
	\end{equation*}
	where $L_0\in W_g(m,n)$, $m,n\in\N$ and $L_j\in\spn W_g(m-j,n+j)$, for $j=1,\dots,m$.
 For $j=0,\dots,m$, define
		\begin{equation*}
			{\mathcal L}_{m,n,j}:=\big(S_g^{q_j+1} T_g\big)^{d_j}\big(S_g^{q_j} T_g\big)^{n+j-d_j},
		\end{equation*} 
		where $q_j=q_j(m,n)$ and $d_j=d_j(m,n)$ are the quotient and the remainder  of the entire division of $m-j$ by $n+j$, respectively, that is, $q_j=\big[\frac{m-j}{n+j}\big]$ and $d_j=m-j-(n+j)q_j$.
	Then 
	\[
	L_g={\mathcal L}_0+\sum_{j=1}^m a_j{\mathcal L}_{m,n,j}\mbox{ on $\H_0(\D)$,}
	\]
	where $a_1,\dots,a_m\in\C$. In particular, when $n$ divides $m$, we have that $q_0\in\N$, $q_j<q_0$, for $j=1,...,m$, and 
	\[
	L=(S_g^{q_0}T_g)^n
	+\sum_{\substack{1\le j\le  m\\q_j=q_0-1}} c_j  {\mathcal L}_{m,n,j}
	+\sum_{\substack{1\le j\le m\\q_j<q_0-1}} c_j {\mathcal L}_{m,n,j}
	\mbox{ on $\H_0(\D)$}.	
\]
	Moreover, if $q_j=q_0-1$, for some $j=1,\dots,m$, then 	$0\le d_j=n-jq_0<n$.
\end{corollary}

\subsection{An approximation result by dilated operators}
In this section we state a pivotal operator approximation result which will be a key tool for the proof of our main theorems. But before doing that we need to recall some basic properties of Bergman spaces induced by radial weights whose proofs will be sketched for the sake of completeness.

For $h\in\H(\D)$ and  $\lambda\in\overline{\D}$, let us consider the 
{\em dilated functions }
\[
h_{\lambda}(z):=h(\lambda z)\qquad(z\in\D).
\] 

Then we have the following result on approximation by dilated functions, which is straightforward.

\begin{proposition}\label{prop:general radial weights} 
Let $\omega$ be a radial weight and $0<p<\infty$. Then:
\begin{enumerate}[label={\sf\alph*)}, topsep=3pt, leftmargin=*,itemsep=3pt] 
\item\label{item:prop:general radial weights1} 
We have the estimate
\begin{equation}\label{eq:UC}
 \sup_{|z|\le r}{|f(z)|^p}\lesssim \frac{\|f\|^p_{A^p_\omega}}{(1-r)\,\widehat{\omega}\left(\frac{1+r}{2}\right)}
\qquad(f\in\H(\D),\,0<r<1),
 \end{equation}
where $\widehat{\omega}(r):=\int_r^1\omega(s)\,ds$.
As a consequence, the convergence in $A^p_\omega$  implies the uniform convergence on compacta, and so $A^p_\omega$ is a complete space. 

\item\label{item:prop:general radial weights2} If $f\in A^p_\omega$ then   $f_\lambda\in A^p_\omega$, for any $\lambda\in\overline{\D}$,  $\|f_{\lambda}\|_{A^p_\omega}\le\|f\|_{A^p_\omega}$, if $\lambda\in\D$, and $\|f_{\lambda}\|_{A^p_\omega}=\|f\|_{A^p_\omega}$, if $\lambda\in\T$.
In addition, 
\begin{equation*}
\lim_{\overline{\D}\ni\lambda\to\zeta}\|f_{\lambda}-f_{\zeta}\|_{A^p_\omega}=0,\quad\mbox{for every $\zeta\in\T$ and $f\in A^p_\omega$.}
\end{equation*}
As a consequence, the polynomials are dense in $A^p_\omega$.
\end{enumerate}
\end{proposition}

It is clear that
\begin{equation*}\label{eq:dilation:MST}
(M_gf)_\lambda=M_{g_\lambda}f_\lambda\qquad
(S_gf)_\lambda=S_{g_\lambda}f_\lambda\qquad
(T_gf)_\lambda=T_{g_\lambda}f_\lambda,
\end{equation*}
and a repeated application of these identities shows that
\begin{equation*}
L_{g_{\lambda}}f_{\lambda}=(L_gf)_{\lambda}\qquad(L_g\in\A_g).
\end{equation*}
The operators $L_{g_{\lambda}}$ are called the {\em dilated operators} of $L_g$.

Now, bearing in mind Proposition~\ref{prop:general radial weights} and following the lines of the proof of \cite[Proposition 4.3]{Aleman:Cascante:Fabrega:Pascuas:Pelaez} we obtain the next result on approximation by dilated operators, which allows us to replace  symbols  in $\H(\D)$ by holomorphic symbols in a neigborhood of $\overline{\D}$.

\begin{proposition}\label{prop:norm:dilations}  
Let $\omega$ be a radial weight, $0<p<\infty$,
 $g\in\H(\D)$ and $L_g\in\A_g$.
If  $L_g\in\BB(A^p_\omega)$  then $L_{g_{\lambda}}\in\BB(A^p_\omega)$ and $\|L_{g_{\lambda}}\|_{A^p_\omega}\lesssim \|L_g\|_{A^p_\omega}$, for any $\lambda\in\overline{\D}$.
Moreover, if ${\displaystyle\varliminf_{r\nearrow1}\|L_{g_r}\|_{A^p_\omega}<\infty}$,
then $L_g\in\BB(A^p_\omega)$ and 
${\displaystyle\|L_g\|_{A^p_\omega}\simeq\varliminf_{r\nearrow1}\|L_{g_r}\|_{A^p_\omega}}$.
 \end{proposition}

\subsection{Analytic tent spaces and Calder\'{o}n formula}
Let $\Gamma(\zeta)$ be the {\em Stolz region with vertex at $\zeta\in\T$} given by 
	\begin{equation*}
		\Gamma(\zeta):=\{z\in\D:|z-\zeta|<2(1-|z|)\},
	\end{equation*}
	and define
	$\Gamma(\zeta)
	:=|\zeta|\Gamma(\frac{\zeta}{|\zeta|})
	=\{z\in\D:|z-\zeta|<2(|\zeta|-|z|)\}$, 
	for $\zeta\in\D\setminus\{0\}$. 
	Then $AT^p_2(\omega)$ is the {\em analytic tent space} of all functions $f\in\H(\D)$ such that 
\begin{equation*}
	\|f\|^p_{AT^p_2(\omega)}:=\int_{\D}\biggl(
	\int_{\Gamma(\zeta)}|f|^2dA\biggr)^{\frac{p}2}
\omega(\zeta)\,\,dA(\zeta)<\infty,
\end{equation*}
and $AT^p_2(\omega,0):=AT^p_2(\omega)\cap\H_0(\D)$.
On the other hand, the {\em non-tangential maximal function} of  $\psi:\D\to\C$ is defined  by
\begin{equation*}
	\M\psi(\zeta):=\sup_{z\in\Gamma(\zeta)}|\psi(z)|
	\qquad(\zeta\in\D\setminus\{0\}).
\end{equation*}
Next result describes the properties of the analytic tent spaces that we need to prove our results.
\begin{proposition}\label{prop:Calderon:formula:tent:spaces}
Let $\omega$ be a radial weight and $0<p<\infty$. Then
\begin{enumerate}[label={\sf\alph*)}, topsep=3pt, leftmargin=*,itemsep=3pt] 
\item\label{item:Calderon:formula:tent:spaces1} The following Calder\'{o}n type formula holds
\begin{equation}\label{normacono}
    \|f\|_{A^p_\omega}^p\simeq \|f'\|^p_{AT^p_2(\omega)}+|f(0)|^p 
\qquad(f\in \H(\D)),
    \end{equation}
where the corresponding constants depend only on $p$ and $\om$.
\item\label{item:Calderon:formula:tent:spaces2} 
The non-tangential maximal operator $\M$ is bounded from $A^p_{\omega}$ to $L^p_{\omega}(\D)$.

\item\label{item:Calderon:formula:tent:spaces3} For any $0<r<1$, we have the estimate
\begin{equation}\label{eq:UC:ATp}
\sup_{|z|\le r}|f(z)|
\lesssim\|f\|_{AT^p_2(\omega)}
\qquad(f\in\H(\D)).
\end{equation}
As a consequence, the convergence in $AT^p_2(\omega)$ implies the uniform convergence on compacta, and so $AT^p_2(\omega)$ is a complete space.

\item\label{item:Calderon:formula:tent:spaces4} The operator $M_z$ is a topological isomorphism from $A^p_{\omega}$ onto $A^p_{\omega}(0)$ and from $AT^p_2(\omega)$ onto $AT^p_2(\omega,0)$.

\item\label{item:Calderon:formula:tent:spaces5} The following estimate holds  \begin{equation*}
	\|h_1h'_2\|_{AT^{p/2}_2(\omega)}
	\lesssim\|h_1\|_{A^p_{\omega}}\|h_2\|_{A^p_{\omega}}
	\qquad(h_1,h_2\in\H(\D)).
\end{equation*}

\end{enumerate}
\end{proposition}

\begin{proof}
Part \ref{item:Calderon:formula:tent:spaces1}  follows by applying the classical Calder\'{o}n formula (see~{\cite[Thm.\ 3]{Calderon} or \cite[Thm.\ 7.4]{Pavlovic}} with $q=2$) to the dilated functions $f_r$, for $0<r<1$, and integrating the resulting estimate against $r\omega(r)\,dr$ along the unit interval 
$(0,1)$. Part \ref{item:Calderon:formula:tent:spaces2} is proved similarly using the boundedness of $\M$ from $H^p$ to $L^p(\T)$ (see \cite[Thm.\ II.3.1]{Garnett}).

Now let us prove \ref{item:Calderon:formula:tent:spaces3}. First note that \eqref{normacono} shows that
\begin{equation}\label{eq:Calderon}
\|f\|_{AT^p_2(\omega)}\simeq\|F\|_{A^p_{\omega}}\qquad(f\in\H(\D)),
\end{equation}
where $F(z)=\int_0^zf(\zeta)\,d\zeta$. This estimate together with Cauchy's formula, \eqref{eq:UC} and \eqref{eq:Calderon} proves \eqref{eq:UC:ATp} as follows:
\begin{equation*}
\sup_{|z|\le r}|f(z)|
\lesssim\sup_{|z|=\frac{1+r}2}|F(z)|
\lesssim\|f\|_{AT^p_2(\omega)}
\qquad(f\in\H(\D)).
\end{equation*}

We next prove \ref{item:Calderon:formula:tent:spaces4}. Let $X$ be either $A^p_{\omega}$ or $AT^p_2(\omega)$. Let $X(0)=A^p_{\omega}(0)$, in the first case, and $X(0)=AT^p_2(\omega,0)$, in the second case.
Recall that  $M_z$ is an algebraic isomorphism from $\H(\D)$ onto $\H_0(\D)$, and 
\begin{equation*}
(M_z^{-1}h_0)(z)=\frac{h_0(z)}z=\int_0^1 h_0'(tz)\,dt\qquad(h_0\in\H_0(\D)).
\end{equation*}
Thus, since $M_z$ is bounded on $X$, we only have to prove that $M_z^{-1}$ is bounded on $X(0)$, that is,
\begin{equation}\label{eqn:prop:top:isom:Apomega:Apomega(0)}
\|h\|_{X}\lesssim\|h_0\|_{X}\qquad(h_0\in\H_0(\D)),
\end{equation}
where $h=M_z^{-1}h_0$. First observe  that
\begin{equation*}
|h(z)|=\frac{|h_0(z)|}{|z|}\le2|h_0(z)|\qquad(\tfrac12\le|z|<1).
\end{equation*}
On the other hand, Cauchy's formula and either \eqref{eq:UC} or \eqref{eq:UC:ATp} give that
\begin{equation*}
\sup_{|z|<\frac12}|h(z)|\lesssim\sup_{|z|<\frac12}|h'_0(z)|\lesssim\sup_{|z|=\frac34}|h_0(z)|\lesssim\|h_0\|_{X}\qquad(h_0\in\H_0(\D)).
\end{equation*}
Therefore 
\begin{equation*}
|h|\lesssim\|h_0\|_{X}\mathbf{1}_{D(0,\frac12)}+|h_0|\mathbf{1}_{\D\setminus\D(0,\frac12)}\qquad(h_0\in\H_0(\D)),
\end{equation*}
where $D(0,\frac12)=\{z\in\D:|z|<\frac12\}$ and  $\mathbf{1}_A$ denotes the indicator or characteristic function of the set $A$. Hence \eqref{eqn:prop:top:isom:Apomega:Apomega(0)} directly follows from this estimate.

Finally, \ref{item:Calderon:formula:tent:spaces5} is proved  using Schwarz inequality,  \ref{item:Calderon:formula:tent:spaces2} and \eqref{normacono} as follows:
\begin{align*}
\|h_1h'_2\|_{T^{p/2}_2(\omega)}
&\le \biggl\{\int_{\D}\biggl(\int_{\Gamma(\zeta)}|h_2'|^2dA\biggr)^{\frac{p}4}(\M h_1(\zeta))^{\frac{p}2}\omega(\zeta)\,dA(\zeta)\biggr\}^{\frac2p}
\\
&\le\|\M h_1\|_{L^p_{\omega}}\|h'_2\|_{T^{p}_2(\omega)} 
\lesssim\|h_1\|_{A^p_{\omega}}\|h_2\|_{A^p_{\omega}}.\qedhere
\end{align*}
\end{proof}

 \section{Proof of  Theorem \ref{thm:compo}}\label{sect:proof:thm11}
 
In order to prove Theorem \ref{thm:compo} we need the following proposition.

\begin{proposition}\label{prop:compo}
Let $\omega$ be a radial weight. Then:
\begin{enumerate}[label={\sf\alph*)}, topsep=3pt, leftmargin=*,itemsep=3pt] 
\item \label{item:compo1} If $T_g\in \BB(A^p_\omega)$, we have that  $\|T^n_g\|_{A^p_\omega}\simeq \|T_g\|_{A^p_\omega}^n$, for any $n\in\N$.
\item \label{item:compo2} 
For any $n\in \N$, $\|T^n_g\|_{A^p_\omega}\simeq \|T^n_g\|_{A^p_\omega(0)}$.
\item \label{item:compo3} 
If $P(T_g)\in\BB(A^p_{\omega}(0))$, for some polynomial $P$ of positive degree,
then $T_g\in\BB(A^p_{\omega})$.
\end{enumerate}
\end{proposition}

Part~{\ref{item:compo1}} of Proposition~{\ref{prop:compo}} is a direct consequence of the following lemma, Proposition~\ref{prop:norm:dilations} and Proposition~\ref{prop:general radial weights} \ref{item:prop:general radial weights2}.

\begin{lemma}
Let $\omega$ be a radial weight, $0<p<\infty$, and $n\in\N$.
Then
\begin{equation}
\label{eqn:lem:T_g-powers:1}
\|T^n_gf\|^2_{A^p_{\omega}}
\lesssim\|T^{n+1}_gf\|_{A^p_{\omega}}\,\|T^{n-1}_gf\|_{A^p_{\omega}}
\quad(f,g\in\H(\D)).
\end{equation} 
Moreover, when $T_g\in\BB(A^p_{\omega})$ we have that
\begin{equation}
\label{eqn:lem:T_g-powers:2}
\|T^n_gf\|^{\frac1n}_{A^p_{\omega}}
\lesssim \|T^{n+1}_gf\|_{A^p_{\omega}}^{\frac1{n+1}}
\qquad(g,f \in\H(\D),\,\|f\|_{A^p_{\omega}}=1).
\end{equation} 
Here, as usual, $T_g^0f=f$, for $g,f\in \H(\D)$.
\end{lemma}

\begin{proof}
By \eqref{normacono},
\begin{equation*}
	\|T^n_gf\|^2_{A^p_{\omega}}
	=\|(T^n_gf)^2\|_{A^{p/2}_{\omega}}
	\simeq\|[(T^n_gf)^2]'\|_{AT^{p/2}_2(\omega)}.
\end{equation*}
Since $[(T^n_gf)^2]'=2(T^{n-1}_gf)\bigl(T^{n+1}_gf\bigr)'$,
estimate~{\eqref{eqn:lem:T_g-powers:1}} follows from  Proposition~{\ref{prop:Calderon:formula:tent:spaces} \ref{item:Calderon:formula:tent:spaces5}}.

 The proof of~{\eqref{eqn:lem:T_g-powers:2}} is done by induction on $n$.
For $n=1$ it is equivalent to~{\eqref{eqn:lem:T_g-powers:1}}. Now let $n>1$. Then the induction hypothesis gives that
\begin{equation*}
\|T^{n-1}_gf\|_{A^p_{\omega}}
\lesssim\|T^n_gf\|^{\frac{n-1}n}_{A^p_{\omega}}
\qquad(g,f\in\H(\D),\,\|f\|_{A^p_{\omega}}=1).
\end{equation*}
This estimate and~{\eqref{eqn:lem:T_g-powers:1}} show that 
\begin{equation*}
\|T^n_gf\|^2_{A^p_{\omega}}
\lesssim\|T^{n+1}_gf\|_{A^p_{\omega}}\, \|T^n_gf\|^{\frac{n-1}n}_{A^p_{\omega}}
\qquad(g,f\in\H(\D),\,\|f\|_{A^p_{\omega}}=1),
\end{equation*}
which is equivalent to~{\eqref{eqn:lem:T_g-powers:2}} since $T_g\in\BB(A^p_{\omega})$, and the proof is complete.
\end{proof}

Next we prove part~{\ref{item:compo2}} of Proposition~{\ref{prop:compo}}.
	Since $\|T_g^n\|_{A_\omega^p(0)}\leq \|T_g^n\|_{A_\omega^p}$, we only have to prove that $\|T_g^n\|_{A_\omega^p}\lesssim \|T_g^n\|_{A_\omega^p(0)}$ and we may assume, as usual, that $g\in\H(\overline{\D})$. Since the pointwise evaluations are bounded on $A^p_{\omega}$ (by Proposition \ref{prop:general radial weights} \ref{item:prop:general radial weights1}), $\Pi_0 f=f-f(0)$ defines a bounded operator on $A_\omega^p$, and so  
$f_0=\Pi_0 f\in A^p_{\omega}(0)$ satisfies that	$\|f_0\|_{A_\omega^p}\lesssim\|f\|_{A_\omega^p}$. Therefore
	\begin{align*}
		\|T^n_gf\|_{A_\omega^p}
		&\lesssim
		(\|T^n_gf_0\|_{A_\omega^p}+|f(0)|\,\|T^n_g1\|_{A_\omega^p})\\
		&\lesssim
		(\|T^n_g\|_{A_\omega^p(0)}\,\|f_0\|_{A_\omega^p}
		+\|T^n_g1\|_{A_\omega^p}\,\|f\|_{A_\omega^p})\\
		&\lesssim
		(\|T^n_g\|_{A_\omega^p(0)}
		+\|T^n_g1\|_{A_\omega^p})\,\|f\|_{A_\omega^p},     
	\end{align*}
	and hence
	$\|T^n_g\|_{A_\omega^p}
		\lesssim
		(\|T^n_g\|_{A_\omega^p(0)}
		+\|T^n_g1\|_{A_\omega^p})$.
		
	Now we want to estimate $\|T^n_g1\|_{A_\omega^p}$ by $\|T^n_g\|_{A_\omega^p(0)}$.
	Since 
	$T^n_g1=\frac{g_0^n}{n!}$ and $T^n_gg_0=\frac{g_0^{n+1}}{(n+1)!}$, where $g_0=\Pi_0 g$  (see \cite[Lemma 3.12]{Aleman:Cascante:Fabrega:Pascuas:Pelaez2}),  H\"older's inequality shows that
	\[
	\|T^n_g1\|_{A_\omega^p}\simeq \|g_0^n\|_{A_\omega^p}
	\lesssim \|g_0^{n+1}\|_{A_\omega^p}^{n/(n+1)}
	\simeq \|T^n_gg_0\|_{A_\omega^p}^{n/(n+1)}.
	\]
	Since $g\in\H(\overline{\D})$, it is clear that 
	$g_0\in A^p_{\omega}(0)$ and we get that 
	\begin{equation*}
		\|T^n_g1\|_{A_\omega^p}\lesssim
		\|T^n_g\|_{A_\omega^p(0)}^{n/(n+1)}\|g_0\|_{A_\omega^p}^{n/(n+1)}.  
	\end{equation*}
By 	H\"older's inequality,
	$\|g_0\|_{A_\omega^p}^{n+1}
	\lesssim \|g_0^{n+1}\|_{A_\omega^p}
	\simeq \|T^n_gg_0\|_{A_\omega^p}\le\|T^n_g\|_{A_\omega^p(0)}\|g_0\|_{A_\omega^p}$,
	which implies that
	$
\|g_0\|_{A_\omega^p}^n\lesssim \|T^n_g\|_{A_\omega^p(0)}$. And that ends the proof of part~{\ref{item:compo2}}.

Finally, we prove part~{\ref{item:compo3}} of Proposition~{\ref{prop:compo}}.
Let  $P(z)=\sum_{k=0}^Na_kz^k$, where $N>0$ and $a_N\ne0$. Then  parts \ref{item:compo1} and \ref{item:compo2}  and Proposition~{\ref{prop:norm:dilations}} give that
\[
c\,|a_N|\|T_{g_r}\|^N_{A_\omega^p(0)}
-\sum_{k=0}^{N-1}|a_k|\|T_{g_r}\|_{A_\omega^p(0)}^k
\le \|P(T_{g_r})\|_{A_\omega^p(0)}\lesssim\|P(T_{g})\|_{A_\omega^p(0)} ,
\]
where $c$ is a positive constant only depending on $N$, $p$, and $\omega$.
 Therefore $\sup_r\|T_{g_r}\|_{A_\omega^p(0)}<\infty$, and using again Proposition~{\ref{prop:norm:dilations}} and parts~{\ref{item:compo1} and \ref{item:compo2}} we conclude that $T_{g}\in\BB(A_\omega^p)$.

\begin{proof}[{\bf Proof of Theorem \ref{thm:compo}}]
It follows the ideas of the proof of~{\cite[Theorem 1.1 (b)]{Aleman:Cascante:Fabrega:Pascuas:Pelaez}.}
Let $L_g$ be a non-trivial $g$-operator which is bounded on $A^p_{\omega}(0)$. 
Then the $ST$-representation \eqref{eq:STrepre} allow us to write $L_g$ as 
\[
L_g=\sum_{k=0}^nS_g^kP_k(T_g)\quad\mbox{ on $\H_0(\D)$},
\] 
where $P_0,\dots,P_n$ are one variable polynomials, and $P_n\ne0$ has degree $m$. Moreover, if $n=0$ then 
$P_0$ has positive degree. In this case,  $L_g=P_0(T_g)$ on $\H_0(\D)$, so $P_0(T_g)\in\BB(A^p_{\omega}(0))$, and therefore Proposition~{\ref{prop:compo} \ref{item:compo3}} gives that $T_{g}\in\BB(A_\omega^p)$. In order to deal with the case $n>0$,
we will use iterated commutators and their main properties as stated 
in \cite[Section~4]{Aleman:Cascante:Fabrega:Pascuas:Pelaez}. Since $P_k(T_{g_r})$ commute  with $T_{g_r}$, we have
\begin{equation*}
	[L_{g_r},T_{g_r}]_n
	= n!\,T_{g_r}^{2n}P_n(T_{g_r})= Q_n(T_{g_r})
\quad\mbox{on $\H_0(\D)$,}
\end{equation*}
where $Q_n$ is a one variable polynomial of degree $N=2n+m>n$,
 so Proposition~{\ref{prop:norm:dilations}} and the binomial formula for iterated commutators implies that
\[
\|Q_n(T_{g_r})\|_{A_\omega^p(0)}=
\|[L_{g_r},T_{g_r}]_n\|_{A_\omega^p(0)}
\lesssim\|L_g\|_{A_\omega^p(0)}\|T_{g_r}\|_{A_\omega^p(0)}^n.
\]

On the other hand,  since $Q_n(z)=\sum_{k=0}^Na_kz^k$, where $a_k\in\C$, taking into account Proposition~{\ref{prop:compo} \ref{item:compo1}-\ref{item:compo2}}, we have 
\[
c\,|a_N|\|T_{g_r}\|^N_{A_\omega^p(0)}
-\sum_{k=0}^{N-1}|a_k|\|T_{g_r}\|_{A_\omega^p(0)}^k
\le \|Q_n(T_{g_r})\|_{A_\omega^p(0)},
\]
where $c$ is a positive constant only depending on $N$, $p$, and $\omega$.
It follows that $\sup_r\|T_{g_r}\|_{A_\omega^p(0)}<\infty$, and so Proposition~{\ref{prop:norm:dilations}} and Proposition~{\ref{prop:compo} \ref{item:compo2}}  show that $T_{g}\in\BB(A_\omega^p)$.
\end{proof}

Observe that, if $T_g$ is bounded on $A^p_{\omega}$, then  $n!\,T_g^n 1=(g-g(0))^n\in A^p_\omega(0)$, for any $n\in\N$, so the operator $P_{N+2}(g-g(0),g(0))\delta_0$, which appears in \eqref{eq:STrepre}, is bounded on $A^p_\omega$. This observation together with Theorem \ref{thm:compo} implies the following
\begin{corollary}
	Any non-trivial $g$-operator $L_g$ is bounded on $A^p_\omega$ if and only if  it is bounded on $A^p_\omega(0)$.
\end{corollary}

\section{Proofs of Theorems \ref{thm:compoq} and \ref{thm:compoq:bis}}
\label{sect:proof:thm13-15:prop16}
\subsection{Proof of Theorem \ref{thm:compoq} \ref{item:compo1}}	
Assume that $L_g\in\BB(A^p_{\omega}(0))$. Then, by Theorem \ref{thm:compo}, $T_g\in \BB(A_\omega^p)$. Since	$[L_g,T_g]_m= m!\,T_g^N$	on $\H_0(\D)$, where $N=2m+n$,
it follows from  Theorem \ref{thm:compo} that 
	\[
		\|T_g\|_{A_\omega^p(0)}^N
		\lesssim \|[L_g,T_g]_m\|_{A_\omega^p(0)} \lesssim \|L_g\|_{A_\omega^p(0)}\|T_g\|_{A_\omega^p(0)}^m,
\]
so $\|T_g\|_{A_\omega^p(0)}^{n+m}\lesssim\|L_g\|_{A_\omega^p(0)}$. If $L_g\in\BB(A^p_{\omega})$ then the above estimate together with Theorem \ref{thm:compo} and Proposition~{\ref{prop:compo} \ref{item:compo2}} give \[
\|T_g\|_{A_\omega^p}^{n+m}\lesssim\|T_g\|_{A_\omega^p(0)}^{n+m}\lesssim\|L_g\|_{A_\omega^p(0)}\le\|L_g\|_{A_\omega^p}.
\]

\subsection{Proof of Theorem \ref{thm:compoq} \ref{item:compo2}} 
We need the following simple lemma.

\begin{lemma}\label{lemma:boundedness:Mg:Sg:Apomega}
Let $\omega$ be a radial weight and $0<p<\infty$. Then:
\begin{enumerate}[label={\sf\alph*)}, topsep=3pt, leftmargin=*,itemsep=3pt] 
\item \label{item:lemma:boundedness:Mg:Sg:Apomega1}
$M_g\in\BB(A^p_{\omega})$ if and only if  $g\in H^{\infty}$, and  
$\|M_g\|_{A^p_{\omega}}\simeq\|g\|_{H^{\infty}}$.

\item  \label{item:lemma:boundedness:Mg:Sg:Apomega2}
 $S_g\in\BB(A^p_{\omega})$ if and only if  $g\in H^{\infty}$, and  
$\|S_g\|_{A^p_{\omega}}\simeq\|g\|_{H^{\infty}}$.
\end{enumerate}
In particular, if $g\in H^{\infty}$ then $T_g\in\BB(A^p_{\omega})$ and 
$\|T_g\|_{A_\omega^p}\lesssim\|g\|_{H^{\infty}}$.
\end{lemma}

\begin{proof}
By Proposition~{\ref{prop:general radial weights} \ref{item:prop:general radial weights1},}
 the space of pointwise evaluations are bounded of $A^p_{\omega}$, therefore the space of  pointwise multipliers of $A^p_{\omega}$ coincides with $H^{\infty}$ 
by \cite[Lemma 11]{Duren:Romberg:Shields} and $\|M_g\|_{A^p_{\omega}}\simeq\|g\|_{H^{\infty}}$.

On the other hand, by Proposition~{\ref{prop:Calderon:formula:tent:spaces} \ref{item:Calderon:formula:tent:spaces1},} $S_g\in\BB(A^p_{\omega})$ means that $g$ is a multiplier on $AT^p_2(\omega)$, so, bearing in mind 
Proposition~{\ref{prop:Calderon:formula:tent:spaces} \ref{item:Calderon:formula:tent:spaces3},} an analogous argument to the previous one gives that
 $\|S_g\|_{A^p_{\omega}}\simeq\|g\|_{H^{\infty}}$.

Finally, if $g\in H^\infty$ then $M_g,S_g\in\BB(A^p_{\omega}$), so $T_g=M_g-S_g-g(0)\delta_0$ is also bounded on $A^p_{\omega}$, by Proposition ~{\ref{prop:general radial weights} \ref{item:prop:general radial weights1}}, and $\|T_g\|_{A_\omega^p}\lesssim\|g\|_{H^{\infty}}$. 
\end{proof}

If $g\in H^\infty$, then 
Lemma~{\ref{lemma:boundedness:Mg:Sg:Apomega}} gives that
\[ 
\|S_g^{m-j}T_g^{j}\|_{A_\omega^p}\le \|S_g\|^{m-j}_{A_\omega^p}\|T_g\|^{j}_{A_\omega^p}\lesssim \|g\|^m_{H^{\infty}},\quad\mbox{ for $0\leq j\le m$,}
\]
so $\|L_g\|_{A_\omega^p}\lesssim \|g\|^m_{H^{\infty}}$.

Now we want to prove the estimate~{$\|g\|_{H^{\infty}}^m\lesssim\|L_g\|_{A_\omega^p}$,} 
	or equivalently $\sup_{0<r<1}\|g_r\|^m_{H^{\infty}}\lesssim\|L_g\|_{A_\omega^p}$, where $g_r(z)=g(rz)$. Assume that $L_g\in\BB(A^p_{\alpha})$. 
	
	By Lemma~{\ref{lemma:boundedness:Mg:Sg:Apomega}}, Proposition~\ref{prop:norm:dilations} and 
Theorem~{\ref{thm:compoq} \ref{item:compo1}, we have that
	\begin{align*}
		\|g_r\|^m_{H^{\infty}}
		&\lesssim\|S^m_{g_r}\|_{A_\omega^p}
		\lesssim
		\|L_{g_r}\|_{A_\omega^p}
		+\sum_{j=1}^m \|S_{g_r}\|^{m-j}_{A_\omega^p}\|T_{g_r}\|^j_{A_\omega^p}
		\\	
		&\lesssim
		\|L_{g}\|_{A_\omega^p}
		+\|T_{g}\|_{A_\omega^p}\sum_{j=1}^m \|S_{g_r}\|^{m-j}_{A_\omega^p}\|T_{g_r}\|^{j-1}_{A_\omega^p}\\
		&\lesssim
	\|L_{g}\|_{A_\omega^p}
		+\|L_{g}\|_{A_\omega^p}^{\frac1m}\,\|g_r\|^{m-1}_{H^{\infty}}.
	\end{align*}
	Hence we conclude that
	$\sup_{0<r<1}\|g_r\|^m_{H^{\infty}}\lesssim\|L_g\|_{A_\omega^p}$, and we are done.

	\subsection{Proof of Theorem \ref{thm:compoq} \ref{item:compo3}} 
We proceed by complete induction on~{$k$.}
 If $k=0$, the estimate  $\|T_g\|_{A^p_{\omega}}\lesssim \|L_g\|^{1/(m+n)}_{A^p_{\omega}(0)}$ follows from 
Theorem \ref{thm:compoq} \ref{item:compoq1}. In particular, this estimate shows that $0<\|L_g\|_{A^p_{\omega}(0)}<\infty$.
	Now assume that $\|S_g^jT_g\|_{A^p_{\omega}}\lesssim \|L_g\|^{(j+1)/(m+n)}_{A^p_{\omega}(0)}$, for $j=0,..., k-1$, and we will prove that
	$\|S_g^kT_g\|_{A^p_{\omega}}\lesssim \|L_g\|^{(k+1)/(m+n)}_{A^p_{\omega}(0)}$.
By  Theorem~{\ref{thm:compoq} \ref{item:compoq1}}, 
$Q_g=L_g^kT_g^{m-nk}$ satisfies that
\begin{equation*}
\|Q_g\|_{A^p_{\omega}(0)}\le\|L_g\|^{k}_{A^p_{\omega}(0)}\|T_g\|^{m-nk}_{A^p_{\omega}(0)}\le \|L_g\|^{k+(m-nk)/(m+n)}_{A^p_{\omega}(0)}=\|L_g\|^{m(k+1)/(m+n)}_{A^p_{\omega}(0)}.
\end{equation*}
Since 
\begin{equation*}
Q_g=L_0+\sum_{j=1}^{km}L_j\quad\mbox{on $\H_0(\D)$},
\end{equation*}
	where $L_0\in W_g(km,m)$, $m,n\in\N$ and $L_j\in\spn W_g(km-j,m+j)$, for $j=1,\dots,m$,
we may apply Corollary  \ref{cor:repreT} to $Q_g$ and obtain that
\begin{equation*}
Q_g=(S_g^kT_g)^m+
        \sum_{j=1}^{km} c_j {\mathcal L}_{km,m,j}\quad\mbox{on $\H_0(\D)$},
\end{equation*}
where $c_j\in\C$. Then, by Theorem \ref{thm:compo} \ref{item:compo1}-\ref{item:compo2}, we have that
	\begin{align*}
		\|S_g^kT_g\|^m_{A^p_{\omega}}&\simeq\|(S_g^kT_g)^m\|_{A^p_{\omega}(0)}\\
		&\le \|Q_g\|_{A^p_{\omega}(0)}
		+\Bigl\{\sum_{\substack{1\le j\le km\\q_j(km,m)=k-1}}
		+\sum_{\substack{1\le j\le km\\ q_j(km,m)<k-1}}
\Bigr\} |c_j| \|  {\mathcal L}_{km,m,j}\|_{A^p_{\omega}(0)}.
	\end{align*}
	If $q_j:=q_j(km,m)=k-1$, for some $1\le j\le km$, then $d_j:=d_j(km,m)<m$.
	In this case,
	\begin{align*}
		\| {\mathcal L}_{km,m,j}\|_{A^p_{\omega}(0)}
		&\le  \|S_g^kT_g\|^{d_j}_{A^p_{\omega}(0)}\|S_g^{k-1}T_g\|^{m+j-d_j}_{A^p_{\omega}(0)}\\
		&\lesssim \|S_g^kT_g\|^{d_j}_{A^p_{\omega}(0)}
\|L_g\|^{k(m+j-d_j)/(m+n)}_{A^p_{\omega}(0)},
	\end{align*}
	with
	\begin{align*}
		k(m+j-d_j)
		&=(k-1)(m+j)+d_j+m+j
-(k+1)d_j\\
		&=km-j+m+j-(k+1)d_j=m(k+1)-(k+1)d_j,
	\end{align*}	
	so
	\quad
	$\|  {\mathcal L}_{km,m,j}\|_{A^p_{\omega}(0)}\lesssim \|S_g^kT_g\|^{d_j}_{A^p_{\omega}(0)}
	\|L_g\|^{(m(k+1)-(k+1)d_j)/(m+n)}_{A^p_{\omega}(0)}$.
	
	If $q_j<k-1$, then
	\begin{align*}
		\|  {\mathcal L}_{km,m,j}\|_{A^p_{\omega}(0)}
		&\lesssim  \|S_g^{q_j+1}T_g\|^{d_j}_{A^p_{\omega}(0)}
\|S_g^{q_j}T_g\|^{m+j-d_j}_{A^p_{\omega}(0)}\\
		&\lesssim \|L_g\|^{\{(q_j+2)d_j
			+(q_j+1)(m+j-d_j)\}/(m+n)}_{A^p_{\omega}(0)}\\
		&= \|L_g\|^{m(k+1)/(m+n)}_{A^p_{\omega}(0)},
	\end{align*}
	From all these estimates, we have 
	\begin{equation*}
		\left(\frac{\|S_g^kT_g\|_{A^p_{\omega}}}{\|L_g\|^{(k+1)/(m+n)}_{A^p_{\omega}(0)}}\right)^m\lesssim 1+\sum_{\substack{1\le j\le km\\q_j
				=k-1}} \left(\frac{\|S_g^kT_g\|_{A^p_{\omega}}}{\|L_g\|^{(k+1)/(m+n)}_{A^p_{\omega}(0)}}\right)^{d_j}.
	\end{equation*}
	Finally, since $d_j<m$, we obtain that  $\|S_g^kT_g\|\lesssim \|L_g\|^{(k+1)/(m+n)}$, which ends the proof. 
\vspace{12pt}

\subsection{Proof of Theorem \ref{thm:compoq:bis} }
By Theorem \ref{thm:compoq} \ref{item:compoq3},
\[
\|S_g^{m/n}T_g\|_{A_\omega^p}\lesssim \|L_g\|^{(m/n+1)/(m+n)}_{A_\omega^p}=\|L_g\|^{1/n}_{A_\omega^p}.
\]
In order to show the opposite estimate, by Corollary \ref{cor:repreT}, it is enough to prove that, for $j=1,\dots,m$, we have that 
$\|{\mathcal L}_j\|_{A_\omega^p}\lesssim \|S_g^{m/n}T_g\|^n_{A_\omega^p}$. 
This estimate is a consequence of Theorem \ref{cor:compoq}. Indeed, since $q_j<q_0=m/n$, applying twice Theorem \ref{cor:compoq} we get that
\begin{align*}
	\|S_g^{q_j+1}T_g\|_{A_\omega^p}&\lesssim \|S_g^{m/n}T_g\|_{A_\omega^p}^{(q_j+2)/(m/n+1)}\quad\mbox{and}\\
	\|S_g^{q_j}T_g\|_{A_\omega^p}&\lesssim \|S_g^{m/n}T_g\|_{A_\omega^p}^{(q_j+1)/(m/n+1)},
\end{align*}
so we obtain that
\begin{equation*}
	\|{\mathcal L}_j\|_{A_\omega^p}\le \|S_g^{q_j+1}T_g\|_{A_\omega^p}^{d_j}
	\|S_g^{q_j}T_g\|_{A_\omega^p}^{n+j-d_j}
	\lesssim \|S_g^{m/n}T_g\|_{A_\omega^p}^\alpha,
\end{equation*}
with 
\begin{align*}
	\alpha&=\Bigl((q_j+2)d_j
	+ (q_j+1)(n+j-d_j)\Bigr)\,\frac1{\frac{m}n+1}\\
	&=(q_j(n+j)+d_j+n+j)\,\tfrac{n}{m+n}=(m-j+n+j)\,\tfrac{n}{m+n}=n.
\end{align*}
This ends the proof of part \ref{item:compoq:bis1}.
Since any $L_g\in W_g(\ell,m,n)$ satisfies \eqref{eq:opLLL}, replacing $m$ by $\ell+m$, part \ref{item:compoq:bis2} directly follows. 
\vspace*{12pt}

As a consequence of the above theorems we obtain the following result.

\begin{proposition}\label{prop:compoq:bis}
Let $L_g$ be a $g$-operator such that
\begin{equation}\label{eqn:prop:compoq:bis}
L_g=a_{2,0}S_g^2+a_{1,0}S_g+a_{1,1}S_gT_g+b_{1,1}T_gS_g+a_{0,2}T_g^2+a_{0,1}T_g\mbox{ on $\H_0(\D)$},
\end{equation}
where the $a_{j,k}$'s are complex numbers.
\begin{enumerate}[label={\sf\alph*)}, topsep=3pt, leftmargin=*,itemsep=3pt] 
\item \label{item:cor:compoq:bis1}
If $a_{2,0}\ne 0$, then $L_g\in\BB(A^p_\omega)$ if and only if $S_g\in\BB(A^p_\omega)$.

\item \label{item:cor:compoq:bis11}
If $a_{2,0}=0$, $a_{1,0}\ne0$ and $a_{1,1}+b_{1,1}=0$, then $L_g\in\BB(A^p_\omega)$ if and only if $S_g\in\BB(A^p_\omega)$.

\item \label{item:cor:compoq:bis2} 
If $a_{2,0}=a_{1,0}=0$ and $a_{1,1}+b_{1,1}\ne 0$, then $L_g\in\BB(A^p_\omega)$ if and only if 
$S_gT_g=\frac 12 T_{g^2}\in\BB(A^p_\omega)$.

\item \label{item:cor:compoq:bis3} 
If $a_{2,0}=a_{1,0}= a_{1,1}+b_{1,1}= 0$, and $a_{0,2}\ne 0$ or $a_{0,1}\ne 0$, then $L_g\in\BB(A^p_\omega)$ if and only if  $T_g\in\BB(A^p_\omega)$.
\end{enumerate}
\end{proposition}

\begin{remark}\label{remark:prop:compoq:bis} For a general radial weight $\omega$, the only case where we don't have a description of the boundedness on $A^p_{\omega}$ of the $g$-operator given by \eqref{eqn:prop:compoq:bis} is $a_{2,0}=0$, $a_{1,0}\ne0$, and  $a_{1,1}+b_{1,1}\ne0$.
 But when $\omega$ is either a radial doubling weight ($\omega\in\DD$) or a rapidly decreasing weight ($\omega\in\W$) the remaining case can be done (see Section \ref{S:radialweights}), and consequently we obtain a description of the bounded $g$-operators which a linear combinations of $1$ and $2$-letter $g$-words.
\end{remark}

\begin{proof}[{\bf Proof of Proposition \ref{prop:compoq:bis}}]
Since $T_gS_g=S_gT_g-T_g^2$ on $\H_0(\D)$, 
\begin{equation*}
L_g=a_{2,0}S_g^2+a_{1,0}S_g+(a_{1,1}+b_{1,1})S_gT_g+(a_{0,2}-b_{1,1})T_g^2+a_{0,1}T_g\quad\mbox{on $\H_0(\D)$},
\end{equation*}
so any of the hypothesis in the proposition implies that $cL_g$ satisfies \eqref{eq:opLL}, for some non-zero constant $c$. Therefore, by  Theorem~\ref{thm:compo}, $T_g\in\BB(A^p_0)$ whenever $L_g\in\BB(A^p_\omega)$, and, in particular, \ref{item:cor:compoq:bis3} follows.  As a consequence, $L_g$ is bounded on $A^p_\omega$ if and only so is the $g$-operator $\widetilde{L}_g=L_g-a_{0,1}T_g-(a_{0,2}-b_{1,1})T_g^2$. Recall that 
\begin{equation*}
\widetilde{L}_g=a_{2,0}S_g^2+a_{1,0}S_g+(a_{1,1}+b_{1,1})S_gT_g\quad\mbox{on $\H_0(\D)$},
\end{equation*}
and so $\widetilde{L}_g$ also satisfies \eqref{eq:opLLL}, up to a non-zero multiplicative constant. Then it is clear that \ref{item:cor:compoq:bis11} and \ref{item:cor:compoq:bis2} hold. 

We finally prove part \ref{item:cor:compoq:bis1}. 
Assume that $a_{2,0}\ne 0$ and $L_g\in\BB(A^p_\omega)$. Then, taking into account that 
\[
S_g^2+2\lambda S_g=S^2_{g+\lambda},\,\, T_{g}=T_{g+\lambda},\,\,
S_gT_g=S_{g+\lambda}T_{g+\lambda}-\lambda T_{g+\lambda}
\,\mbox{ on $\H_0(\D)$,}
\]
 for any 
$\lambda\in\C$, we may also assume that $a_{1,0}=0$.
 Then Theorem \ref{thm:compoq} shows that $S_g\in\BB(A^p_\omega)$. The converse is clear because $S_g\in\BB(A^p_\omega)$ implies $T_g\in\BB(A^p_\omega)$, by Lemma \ref{lemma:boundedness:Mg:Sg:Apomega}.
 And that ends the proof. 
\end{proof}

\begin{remark} 
As a consequence of the above results we  show the full characterization of the boundedness on $A^p_{\omega}$ of any two-letter $g$-word.

Obviously, the formula $S_gT_g=T_gM_g=\frac{1}{2}T_{g^2}$ gives that $S_gT_g=T_gM_g\in\BB(A^p_{\omega})$ if and only if $T_{g^2}\in\BB(A^p_{\omega})$. Next, since $M_g^2=M_{g^2}$ and  $S_g^2=S_{g^2}$,  Lemma \ref{lemma:boundedness:Mg:Sg:Apomega} shows that 
any of the operators $M_g^2$ and $S_g^2$ is bounded on $A^p_{\omega}$ if and only if $g\in H^{\infty}$.

Finally, bearing in mind that
\begin{equation*}
M_gT_g= S_gT_g+T^2_g\mbox{ and } T_gS_g=S_gT_g-T^2_g
\mbox{ on $\H_0(\D)$,}
\end{equation*} 
Proposition \ref{prop:compoq:bis} shows that any of those two operators are bounded on $A^p_\omega$ if and only if $T_{g^2}\in\BB(A^p_{\omega})$.
\end{remark}

\section{Boundedness of single analytic paraproducts}
\label{sect:boundedness:single:analytic:paraproducts}
In this section we will give a characterization of the boundedness of $M_g$, $S_g$, and $T_g$, by using analytic tent spaces, a Calder\'{o}n type formula, and spaces of pointwise multipliers.

If $X,Y$ are Banach or quasi Banach spaces,
$\Mult(X,Y)$ denotes the space of pointwise multipliers from $X$ to $Y$, and $\Mult(X):=\Mult(X,X)$. Recall that 
$\|g\|_{\Mult(X,Y)}=\|M_g\|_{X\to Y}$, for any $g\in\H(\D)$.

\begin{proposition}\label{prop:boundedness:Mg:Sg:Tg}
Let $\omega$ be a radial weight and $0<p<\infty$. Then:
\begin{enumerate}[label={\sf\alph*)}, topsep=3pt, leftmargin=*,itemsep=3pt] 
\item \label{item:prop:boundedness:Mg:Sg:Tg1}
$M_g\in\BB(A^p_{\omega})$ if and only if  $g\in H^{\infty}$, and  
$\|M_g\|_{A^p_{\omega}}\simeq \|M_g\|_{A^p_{\omega}(0)}\simeq\|g\|_{H^{\infty}}$.

\item  \label{item:prop:boundedness:Mg:Sg:Tg2}
 $S_g\in\BB(A^p_{\omega})$ if and only if  $g\in H^{\infty}$, and  
$\|S_g\|_{A^p_{\omega}}\simeq \|S_g\|_{A^p_{\omega}(0)}\simeq\|g\|_{H^{\infty}}$.

\item \label{item:prop:boundedness:Mg:Sg:Tg3}
$T_g\in\BB(A^p_{\omega})$ if and only if  
$g'\in\Mult(A^p_{\omega},AT^p_2(\omega))$,  and 
\[
\|T_g\|_{A^p_{\omega}}\simeq\|T_g\|_{A^p_{\omega}(0)}
\simeq\|g'\|_{\Mult(A^p_{\omega},AT^p_2(\omega))}.
\]
Moreover, if $g\in BMOA$, then $T_g\in\BB(A_\omega^p)$ and $\|T_g\|_{A_\omega^p}\lesssim \|g\|_{BMOA}$.
\end{enumerate}
\end{proposition}

\begin{proof}
First note that 
Lemma \ref{lemma:boundedness:Mg:Sg:Apomega} shows the estimate 
$\|M_g\|_{A^p_{\omega}}\simeq\|g\|_{H^{\infty}}\simeq\|S_g\|_{A^p_{\omega}}$, so in order to complete the proofs of parts \ref{item:prop:boundedness:Mg:Sg:Tg1} and \ref{item:prop:boundedness:Mg:Sg:Tg2}
we just have to show that $\|M_g\|_{A^p_{\omega}(0)}\simeq\|g\|_{H^{\infty}}\simeq\|S_g\|_{A^p_{\omega}(0)}$.

Since $M_g(zf(z))=M_{zg(z)}f$, for any $f,g\in\H(\D)$,
Proposition~{\ref{prop:Calderon:formula:tent:spaces} \ref{item:Calderon:formula:tent:spaces4}} shows that  $M_g\in\BB(A^p_{\omega}(0))$ if and only if $M_{zg(z)}\in\BB(A^p_{\omega})$,
and $\|M_g\|_{A^p_{\omega}(0)}
\simeq \|M_{zg(z)}\|_{A^p_{\omega}}$.
Moreover, by Lemma \ref{lemma:boundedness:Mg:Sg:Apomega}, we have that $M_{zg(z)}\in\BB(A^p_{\omega})$ if and only if $zg(z)\in H^{\infty}$, that is
$g\in H^{\infty}$, and 
$\|M_{zg(z)}\|_{A^p_{\omega}}\simeq\|zg(z)\|_{\infty}\simeq\|g\|_{H^\infty}$. 

 If $g\in H^{\infty}$, then we have already proved that 
$\|S_g\|_{A^p_{\omega}(0)}\le\|S_g\|_{A^p_{\omega}}\lesssim\|g\|_{H^\infty}$.
In order to prove the converse recall that, by Proposition~{\ref{prop:general radial weights} \ref{item:prop:general radial weights1},} $\Pi_0f:=f-f(0)$ defines a bounded operator from $A^p_{\omega}$ to $A^p_{\omega}(0)$. It follows that if $S_g\in\BB(A^p_{\omega}(0))$,  then $S_g=S_g\circ\Pi_0 \in\BB(A^p_{\omega})$ and 
$\|S_g\|_{A^p_{\omega}}\lesssim \|S_g\|_{A^p_{\omega}(0)}$, so  $g\in H^{\infty}$ and  $\|g\|_{H^\infty}\lesssim\|S_g\|_{A^p_{\omega}(0)}$.
Hence we have just proved parts \ref{item:prop:boundedness:Mg:Sg:Tg1} and \ref{item:prop:boundedness:Mg:Sg:Tg2}.

The first part of \ref{item:prop:boundedness:Mg:Sg:Tg3} follows from  \eqref{normacono} and 
Proposition~{\ref{prop:compo} \ref{item:compo2}.}
 Finally, assume that $g\in BMOA$. Then, by integrating in polar coordinates and taking into account the classical estimate $\|T_g\|_{H^p}\simeq \|g\|_{BMOA}$ (see \cite{Aleman:Siskakis:1} and \cite{Aleman:Cima}), we have
 \begin{align*} \|T_g f\|_{A_\omega^p}^p&\simeq \int_0^1 r \|(T_gf)_r\|_{H^p}^p \omega(r) dr=  \int_0^1 r \|T_{g_r}f_r\|_{H^p}^p \omega(r) dr\\&\lesssim	\int_0^1 r \|g_r\|_{BMOA}^p\|f_r\|_{H^p}^p \omega(r) dr\\& \lesssim \|g\|_{BMOA}^p \int_0^1 r \|f_r\|_{H^p}^p \omega(r) dr\simeq \|g\|_{BMOA}^p \|f\|_{A_\omega^p}^p,
 \end{align*}
so $\|T_g\|_{A_\omega^p}\lesssim \|g\|_{BMOA}$. And that ends the proof.
\end{proof}

As a consequence of Theorem \ref{cor:compoq} and 
Proposition {\ref{prop:boundedness:Mg:Sg:Tg} \ref{item:prop:boundedness:Mg:Sg:Tg3}} we obtain the following result.

\begin{corollary} 
	If $\omega$ is a radial weight and $0<p<\infty$, then
\begin{equation*}	
	\|(g^j)'\|^{\frac1j}_{Mult(A^p_\omega,AT^p_2(\omega))} \lesssim \|(g^k)'\|_{Mult(A^p_\omega,AT^p_2(\omega))}^{\frac1k}
\qquad(1\le j\le k).
\end{equation*}
\end{corollary}

\section{Further results for some classes of radial weights}\label{S:radialweights}

We begin this section obtaining a further result on the composition of analytic paraproducts acting 
on a Bergman space $A^p_\omega$
which can be applied to several classes of radial weights.

\begin{theorem}\label{th:main5}
Let $\omega$ be a radial weight and $0<p<\infty$. Assume that there is $\rho_0\in [0,1)$ such that for any $\xi\in\D\setminus D(0,\rho_0)$ there exists
$K_\xi\in\H(\D)$ with the following properties:
\begin{enumerate}[label={\sf(\alph*)}, topsep=3pt, leftmargin=48pt,itemsep=3pt] 
\item\label{item:th:main5:1} $\|K_\xi\|_{A^p_\omega}=1$.
\item\label{item:th:main5:2} $\lim_{|\xi|\to 1^-} |K_\xi(z)|=0$ uniformly on compact subsets of $\D$.
\item\label{item:th:main5:3}  
${\displaystyle
\lim_{\xi\to\zeta}\int_{\D\setminus D(\zeta,\delta)}|K_\xi|^p\omega\,dA=0}$,  for any $\delta>0$ and $\zeta\in\mathbb{T}$.
\end{enumerate}
Then, if $L_g$ is a $g$-operator written in the form~\eqref{eq:STrepre} such that $P_{N+1}\neq 0$, $L_g$ is bounded on $A^p_\omega$ if and only if
$g\in H^\infty$.
\end{theorem}
The following two lemmas will be used  in the proof of Theorem~\ref{th:main5}.
We start up with a straightforward  approximation identity type result.
	\begin{lemma}\label{le:techinfty}
		Let $\omega$ be a radial weight and $0<p<\infty$ such that the properties \ref{item:th:main5:1} and \ref{item:th:main5:3} of Theorem~{\textup{\ref{th:main5}}} hold. Then any continuous function $F$ on $\overline{\D}$ satisfies that
		\begin{equation*}
			\lim_{\xi\to \zeta}\int_{\D} |K_\xi|^p F\omega\,dA
            =F(\zeta),\quad\mbox{for every $\zeta\in\T$}.
		\end{equation*}
	\end{lemma}
	
	\begin{lemma}\label{le:Tgcompacto}
		Let $\omega$ be a radial weight and $0<p<\infty$.  If $g'\in H^\infty$, then $T_g: A^p_\omega\to A^p_\omega$ is compact.
	\end{lemma}
	\begin{proof}
		Let
		$\{ f_j\}$ be a bounded sequence in
		$A^p_\omega$  such that $f_j\to0$ uniformly on compact subsets of $\D$. By \cite[Lemma~3.7]{Tjani}, we only have to prove that 
$\lim_{j\to \infty} \|T_gf_j\|_{A^p_{\omega}}=0$.
		By Proposition~{\ref{prop:Calderon:formula:tent:spaces} \ref{item:Calderon:formula:tent:spaces1}-\ref{item:Calderon:formula:tent:spaces2},} for any $0<r<1$ we have that
\begin{align*}
\|T_gf_j\|^p_{A^p_\omega}  
&\simeq \int_{\D}\biggl(
				\int_{\Gamma(\zeta)}|g'f_j|^2 dA\biggr)^{\frac{p}2}
				\omega(\zeta)\,\,dA(\zeta)
				\\
 &\lesssim \sup_{|z|\le r}|f_j(z)|^p \int_{\D}\biggl(
				\int_{\Gamma(\zeta)\cap\overline{D(0,r)}}|g'|^2 dA\biggr)^{\frac{p}2}
				\omega(\zeta)\,\,dA(\zeta)\\
&\quad+ \|g'\|^p_{\infty} \int_{\D}\biggl(
				\int_{\Gamma(\zeta)\setminus \overline{D(0,r)}}|f_j|^2 dA\biggr)^{\frac{p}2}
				\omega(\zeta)\,\,dA(\zeta)
				\\ 
& \lesssim\|g'\|^p_{\infty}\sup_{|z|\le r}|f_j(z)|^p+\|g'\|^p_{\infty}(1-r)^{p/2}
                \|\M f_j\|_{L^p_\omega}^p
			\\ 
& \lesssim\|g'\|^p_{\infty}
\biggl(\sup_{|z|\le r}|f_j(z)|^p+(1-r)^{p/2}
                \|f_j\|_{A^p_\omega}^p\biggr).
		\end{align*}
		Therefore, taking into account that $f_j\to0$ uniformly on compacta and $\sup_{j} \|f_j\|_{A^p_\omega}<\infty$, the above inequality shows that $\lim_{j\to \infty}\|T_gf_j\|_{A^p_\omega}=0$, and that finishes the proof.
	\end{proof} 
	
	\begin{proof}[{\bf Proof of Theorem~\ref{th:main5}}] We will follow the lines of the proof of \cite[Theorem 1.2~a)]{Aleman:Cascante:Fabrega:Pascuas:Pelaez}.
		If $g\in H^\infty$, then $S_g,T_g\in\BB(A^p_{\omega})$, by Proposition~\ref{prop:boundedness:Mg:Sg:Tg}, and so
 $L_g\in\BB(A^p_{\omega})$. Conversely, assume that  $L_g\in \BB(A^p_{\omega})$ and apply Proposition~{\ref{prop:norm:dilations}} to conclude that for $r\in (0,1)$, we have  $L_{g_r}\in \BB(A^p_{\omega})$ with $\|L_{g_r}\|_{A^p_\omega}\lesssim\|L_{g}\|_{A^p_\omega}$. From~\eqref{eq:STrepre} we have that
	    \[
		L_{g_r}=\sum_{k=0}^NS_{g_r}^kT_{g_r}P_k(T_{g_r})\Pi_0+S_{g_r}P_{N+1}(S_{g_r})+P_{N+2}(g_r-g(0),g(0))\,\delta_0,
		\]
for any $r\in (0,1)$, and, by Lemma~\ref{le:Tgcompacto}, we see that all the operators on the right are compact, except 
		\[
		S_{g_r}P_{N+1}(S_{g_r}) = S_{g_rP_{N+1}(g_r)}= M_{g_rP_{N+1}(g_r)}-T_{g_rP_{N+1}(g_r)}-(g_rP_{N+1}(g_r))(0)\delta_0.
		\]
		So, by Lemma~\ref{le:Tgcompacto}, we conclude that
		\[
		L_{g_r}=M_{g_rP_{n+1}(g_r)}+J,
		\]
		where $J$ is compact.
		
		Now, for any $\xi\in\D\setminus D(0,\rho_0)$, 
		consider the functions $K_\xi$ of the statement. Then, putting together  hypothesis \ref{item:th:main5:1} and \ref{item:th:main5:2}, and \cite[Lemma~3.7]{Tjani}, we have  
		$\|JK_\xi\|_{A^p_\omega}\to 0$.
		On the other hand, note that if $G_r=g_rP_{n+1}(g_r)$ then
		\[
		\|M_{G_r} K_\xi\|_{A^p_\omega}^p
		=\int_{\D} |K_\xi(z)|^p\,|G_r(z)|^p \omega(z)\,\,dA(z)
		\quad(\xi\in\D\setminus D(0,\rho_0)).
		\]
		Thus, since $|G_r|=|g_rP_{n+1}(g_r)|$ is continuous on $\overline{\D}$, by Lemma~\ref{le:techinfty}
		\[
		\lim_{\xi\to\zeta}\| M_{g_rP_{n+1}(g_r)} K_\xi\|_{A^p_\omega}^p=|g_rP_{n+1}(g_r)(\zeta)|^p\qquad (\zeta\in \T).
		\]
		Altogether we get 
		\[
		|g_rP_{n+1}(g_r)(\zeta)|^p=\lim_{\xi\to\zeta}\|L_{g_r}K_\xi\|^p_{A^p_\omega}\le \|L_g\|^p_{A^p_\omega}\lim_{\xi\to\zeta}\| K_\xi\|^p_{A^p_\omega}= \|L_g\|^p_{A^p_\omega},
		\]
		for all $\zeta\in \T$ and $0<r<1$, which implies that $gP_{n+1}(g)\in H^\infty$, and so $g\in H^{\infty}$, by \cite[Lemma 4.5]{Aleman:Cascante:Fabrega:Pascuas:Pelaez}. Thus the proof is finished.
	\end{proof}

We remark that, in the next sections, Theorem~\ref{th:main5} will allow us  to complete the remaining open case in Proposition \ref{prop:compoq:bis} (see Remark \ref{remark:prop:compoq:bis}) for two classes of radial weights, which have drawn a lot of attention in the recent years.

\subsection{Radial doubling weights}
Recall that the class $\DD$ of {\em radial upper doubling weights} is composed of all radial weights $\omega$ such that
 $\widehat{\om}(r)\le
C\widehat{\om}(\frac{1+r}{2})$, for some constant $C=C(\omega)>1$ and all $0\le r <1$. On the other hand, the class $\Dd$ of {\em radial lower doubling weights} is composed of all radial weights $\omega$ such that such that $\widehat{\omega}(r)\le C\int_{r}^{r+\frac{1-r}{K}}\omega(s)\,ds$, for some constants $K>1$ and $C>0$, and  for any $0\le r<1$. The class of {\em radial doubling weights} is
$\DDD=\DD\cap\Dd$. 
If $\omega$ is radial weight,  the Littlewood-Paley type formula 
 \begin{equation*}
 \|f\|_{A^p_\om}^p\simeq|f(0)|^p+\int_\D|f'(z)|^p(1-|z|)^p\om(z)\,dA(z)  \quad(f\in\H(\D))
 \end{equation*}
holds  if and only if $\omega$ is a radial doubling weight 
(see \cite[Theorem~5]{PR21}). This result is a key to prove that $\TT(A^p_\omega)=\B$ if $\omega\in\DDD$,  see \cite[Proposition~6.1 and Theorem~6.3]{PelSum14}. Consequently,
Theorem~\ref{th:powerBloch} implies that $\TT(A^p_\omega)$ 
satisfies the radicality property if $\omega\in\DDD$.
However the situation is more involved for $\om\in\DD\setminus\DDD$,
 because for each $p\neq 2$ there are radial upper doubling weights $\omega$ such that a Littlewood-Paley type formula \eqref{eq:NOL-P} does not hold for
any radial function $\varphi$  (see \cite[Proposition~4.3]{Pelaez:Rattya:Memoirs} or \cite[Proposition~3.7]{PelSum14}). Therefore, for such weights we have to work with the  Calder\'on type formula \eqref{normacono} to obtain 
an equivalent norm to $\|\cdot\|_{A^p_\omega}$ in terms of the derivative. In order  
 to give a  geometric description of the space $\TT(A^p_\omega)$, $\om\in\DD$,  we introduce the 
 space
$\CC^{1}(\om^\star)$ of  $g\in\H(\D)$ such
that
    \begin{equation*}
    \|g\|^2_{\CC^{1}(\om^\star)}=|g(0)|^2+\sup_{S}\frac{\int_{S}|g'(z)|^2\om^\star(z)\,dA(z)}
    {\om\left(S\right)}<\infty,
    \end{equation*}
    where\begin{equation*}\label{omega:star}
		\omega^\star(z)=\int_{|z|}^1 s\omega(s)\log\frac{s}{|z|}\,ds
		\qquad(z\in\D\setminus\{0\})
	\end{equation*} and $S$ runs over all Carleson squares in $\D$.
If we consider the measure $d\mu_{g,\omega^\star}= |g'(z)|^2\om^\star(z)\,dA(z)$,   a byproduct of \cite[Theorem~3.3]{PelSum14} (see also \cite[Theorem~2.1]{Pelaez:Rattya:Memoirs})   gives that
 the identity operator
$Id: A^p_\omega\to L^p(\mu_{g,\omega^\star})$ is bounded if and only if $g\in \CC^{1}(\om^\star)$, and 
$\|Id\|^p_{A^p_\omega\to L^p(\mu_{g,\omega^\star})}\simeq \|g-g(0)\|^2_{\CC^{1}(\om^\star)}$, for any $\omega\in\DD$.
Bearing in mind this fact, \cite[Theorem~4.1]{Pelaez:Rattya:Memoirs} or \cite[Proposition~6.4]{PelSum14}, and a careful inspection of their proof, we get the following result. 

\begin{lettertheorem}\label{th:integralq=p} 
 If $\om\in\DD$, $g\in\H(\D)$ and $0<p<\infty$, then $T_g\in\BB(A^p_\om)$ if and only if $g\in \CC^1(\om^\star)$.
 Moreover, $\|T_g\|_{ A^p_\omega} \simeq \|g-g(0)\|_{\CC^{1}(\om^\star)}.$
\end{lettertheorem}
    
If $\omega\in\DD$ then $C^1(\omega^\star)\subset \B$. 
Moreover, a calculation together with the proof 
of  \cite[Theorem~5.1 (B)-(C)]{Pelaez:Rattya:Memoirs} implies that $C^1(\omega^\star)= \B$ if and only if $\omega\in \DDD$.

Theorem~\ref{cor:compoq} together with Theorem~\ref{th:integralq=p}
 gives an operator theoretical proof of the following result.
  
\begin{corollary}
Let $\omega\in\DD$ and $0<p<\infty$. Then $\TT(A^p_\omega)=\CC^1(\omega^\star)$ has the radicality property. Moreover, if
 $g\in \H(\D)$ and $n\in\N$ satisfy that $g^{n}\in\CC^1(\omega^\star)$, then $g^m\in\CC^1(\omega^\star)$ for $m\in\N$ $m<n$, and
\[
\| g^m-g^m(0)\|^{\frac{1}{m}}_{\CC^1(\omega^\star)}\lesssim \| g^n -g^m(0)\|^{\frac{1}{n}}_{\CC^1(\omega^\star)}.
\]
\end{corollary}
We recall that the space
$\CC^1(\omega^\star)$ is not necessarily
conformally invariant when $\omega\in\DD\setminus\DDD$ (see \cite[Proposition 5.4]{Pelaez:Rattya:Memoirs} or \cite[Proposition 6.2]{PelSum14}).

Now we will see that Theorem~\ref{th:main5} can be applied to the class $\DD$.

\begin{corollary}
Let $\omega\in\DD$, $g\in\H(\D)$ and $0<p<\infty$.
Then, if $L_g$ is a $g$-operator written in the form  \eqref{eq:STrepre} such that $P_{N+1}\neq 0$, $L_g$ is bounded on $A^p_\omega$ if and only if
$g\in H^\infty$.
\end{corollary}

\begin{proof}
 Since $\omega\in\DD$, Lemma 2.1 in \cite{PelSum14} and its proof show that there is a constant $\beta=\beta(\omega)>0$ which satisfies
\begin{equation}
\label{eqn2} 
\widehat{\omega}(r)
\lesssim\left(\frac{1-r}{1-t}\right)^{\beta}\widehat{\omega}(t)
\quad(0\le r\le t<1)
\end{equation}
and 
\begin{equation}
\label{eqn7} 
\int_{\D}\frac{\omega(z)\,dA(z)}{|1-\overline{\xi}z|^{\eta+1}}
\simeq\frac{\widehat{\omega}(|\xi|)}{(1-|\xi|)^{\eta}}
\quad(\xi\in\D),
\end{equation}
for all $\eta>\beta$.
Choose $\eta>\beta$, and, for each $\xi\in\D$, consider the function
\[ h_\xi(z)=\left(\frac{(1-|\xi|)^\eta}{\widehat{\omega}(|\xi|)(1-\overline{\xi}z)^{\eta+1}}\right)^{\frac{1}{p}}\qquad(z\in\D),
\]
which clearly belongs to $A^p_{\omega}$. 
We will complete the proof by checking that the functions $K_\xi=h_\xi/\|h_\xi\|_{A^p_\omega}$ satisfy the hypotheses \ref{item:th:main5:1}, \ref{item:th:main5:2} and \ref{item:th:main5:3} of Theorem~\ref{th:main5}.
It is clear that \ref{item:th:main5:1} holds. Now \eqref{eqn7} shows that $\|h_\xi\|^p_{A^p_\omega}\simeq1$.
So, by \eqref{eqn2}, for any $0<r<1$ we have that
\[
|K_\xi(z)|^p\lesssim|h_\xi(z)|^p
\le\frac{(1-|\xi|)^\eta}{\widehat{\omega}(|\xi|)(1-r)^{\eta+1}}
\lesssim \frac{(1-|\xi|)^{\eta-\beta}}{\widehat{\omega}(0)(1-r)^{\eta+1}}
\quad(\xi\in\D,|z|\le r),
\]
which implies that $\{K_\xi\}_{\xi\in\D}$ satisfies \ref{item:th:main5:2}.
Next, take $\zeta\in \T$ and $\delta>0$. If $|z-\zeta|\ge \delta$ and $|\xi-\zeta|<\frac{\delta}{2}$, then
$|1-\overline{\xi}z|\ge \frac{\delta}{2}$, so
 using again \eqref{eqn2} we obtain 
\[
|K_\xi(z)|^p\lesssim|h_\xi(z)|^p
\lesssim \frac{(1-|\xi|)^{\eta-\beta}}{\widehat{\omega}(0)}
\quad(|z-\zeta|\ge \delta,\,|\xi-\zeta|<\tfrac{\delta}{2}).
\]
Therefore
\[
\int_{\D\setminus D(\zeta,\delta)}|K_\xi|^p\omega\,dA
\lesssim \frac{(1-|\xi|)^{\eta-\beta}}{\widehat{\omega}(0)}\int_{\D}\omega\,dA,
\]
and hence $\{K_\xi\}_{\xi\in\D}$ satisfies \ref{item:th:main5:3}, which  ends the proof.
\end{proof}

\subsection{Rapidly decreasing radial weights}

\begin{definition}
A radial weight $\omega$ is {\em rapidly decreasing} if it satisfies the following conditions:

\begin{enumerate}[label={\sf(\alph*)}, topsep=3pt, leftmargin=32pt,itemsep=3pt] 
\item $\omega(z)=e^{-\varphi(z)}$,
where $\varphi\in C^2(\D)$ is a radial function such  that
$\Delta\varphi(z)\ge B_{\varphi}>0$ for some positive constant
$B_{\varphi}$ depending only on the function $\varphi$. Here
$\Delta$ denotes the standard Laplace operator.
 
\item $\left(\Delta\varphi(z)\right)^{-1/2}\simeq\tau(z)$, where
$\tau(z)$ is a radial positive function that decreases to $0$ as
$|z|\rightarrow 1^{-}$, and $\lim_{r\to 1^-}\tau'(r)=0$.

\item There exists  a constant $C>0$  such that either
$\tau(r)(1-r)^{-C}$ increases for $r$ close to $1$ or
 $$\lim_{r\to 1^-}\tau'(r)\log\frac{1}{\tau(r)}=0.$$
\end{enumerate}
\end{definition}
The class of rapidly decreasing weights is denoted by $\mathcal{W}$.
This
class  does not include the standard weights, but it  includes  the exponential type
weights
\begin{displaymath}
w_{\alpha}(r)=\exp
\left(\frac{-c}{(1-r)^\alpha}\right),\quad \mbox{for  $c,\alpha> 0$,}
\end{displaymath}
and the double exponential type weights
\begin{displaymath}
w(r)=\exp
\left(\exp
\left(\frac{-c}{1-r}\right)\right),\quad \mbox{for  $c>0$.}
\end{displaymath}

Despite  Proposition~{\ref{prop:Calderon:formula:tent:spaces} \ref{item:Calderon:formula:tent:spaces1}} provides an equivalent norm to $\|\cdot\|_{A^p_\omega}$ in terms of a Calder\'on type formula, when we are interested in obtaining an equivalent norm in terms of the first derivative, it is more convenient to deal with a Littlewood-Paley type formula  when $\omega\in\W$. In fact,  by  \cite[(9.3)]{CP},
for any $p\in(0,\infty)$ and $\omega=e^{-\varphi}\in\W$,
\begin{equation*}
\|f\|^p_{A^p_\omega}
\simeq |f(0)|^p+\int_{\D}|f'(z)|^p\omega(z)\left(\frac{1}{1+\varphi'(|z|)}\right)^p\,dA(z).
\end{equation*}
This Littlewood-Paley type formula together with the existence of $\delta>0$ small enough such that (see \cite[Lemma~32(d)]{CP})
\[
\varphi'(|z|)\simeq \varphi'(|a|)\quad(a\in\D,\,\,z\in D\bigl(a, \delta(\Delta\varphi(a))^{-1/2})),
\]
allows to omit the hypotheses (6) in \cite[Theorem~2]{Pau-Pelaez:JFA2010} and to mimick its proof to obtain the following result,
which was already proved in \cite[Section 9]{CP}.

\begin{lettertheorem}\label{th:integralq=pW} 
 Let $\omega=e^{-\varphi}\in\W$, $g\in\H(\D)$ and $0<p<\infty$. Then $T_g\in\BB(A^p_\om)$ if and only if 
 $\rho(g,\varphi)=\sup_{z\in\D}\frac{1}{1+\varphi'(|z|)}|g'(z)|<\infty.$
 Moreover, $\|T_g\|_{ A^p_\omega\to A^p_\omega} \simeq \rho(g,\varphi).$
\end{lettertheorem}
If $\omega=e^{-\varphi}\in\W$ we write $\B_\varphi=\{g\in\H(\D):\rho(g,\varphi)<\infty \}$.
Then we have that Theorem~\ref{cor:compoq}  together with Theorem~\ref{th:integralq=pW}  provides a proof of the following result.

\begin{corollary}
Let be $\omega=e^{-\varphi}\in\W$ and  and $0<p<\infty$, then $\TT(A^p_\omega)=\B_\varphi$ satisfies the radicality property. Moreover, if
 $g\in \H(\D)$ and $n\in\N$  satisfy that $g^{n}\in \B_\varphi$,  then $g^m\in \B_\varphi$
for $m\in\N$, $m<n$, and  
$$\rho(g^m,\varphi)^{\frac{1}{m}}\lesssim \rho(g^n,\varphi)^{\frac{1}{n}}.$$
\end{corollary}

Unlike  $BMOA$ or $\B$,
the space
$\B_\varphi$ is not 
conformally invariant when $\omega=e^{-\varphi}\in\W$. Indeed, since $\lim_{r\to 1^-}\frac{1}{(1-r)\varphi'(r)}=0$ \cite[Lemma~32(a)]{CP}, the classical Bloch space $\B$ is strictly contained in $\B_\varphi$, if $\omega=e^{-\varphi}\in\W$. On the other hand, it is well-known \cite{RB} that $\B$ is the largest   "reasonable" space $X$ equipped with a conformally invariant seminorm $\rho$. Consequently, bearing in mind that $\B_\varphi$ is a "reasonable" space
(see \cite{RB}), we deduce that $\B_\varphi$ is not 
conformally invariant when $\omega=e^{-\varphi}\in\W$. 

We obtain the
 following result from Proposition~\ref{prop:boundedness:Mg:Sg:Tg} and Theorem~\ref{th:integralq=pW}.

\begin{corollary}
Let $\omega=e^{-\varphi}\in\W$, $g\in\H(\D)$, $0<p<\infty$ and $$\Omega_p(z)=\omega(z)\left(\frac{1}{1+\varphi'(|z|)}\right)^p.$$  
Then $\Mult(A^p_{\omega},A^p_{\Omega_p})=\Mult(A^p_{\omega},AT^p_2(\omega))=\{g: G(z)=\int_0^z g\in \B_\varphi  \}$.
\end{corollary}

Finally, we will apply Theorem~\ref{th:main5} to the class $\W$. With this aim, we recall the next result.

\begin{lettertheorem}[{\cite[Lemma~3.1 and Corollary 1]{Pau-Pelaez:JFA2010}}] \label{le:Borichev}
Assume that $0< p<\infty$,\, $n\in\mathbb{N}$ with
$np\ge 1$ and $\omega\in \mathcal{W}$. Then there is a number 
$\rho_0\in (0,1)$ and a family  $\{F_{\xi}:\xi\in\D,\,
|\xi|\geq\rho_0\}$ of analytic functions on $\D$ satisfying the following estimates:
\begin{align}\label{BL1}
|F_{\xi}(z)|^p\,w(z)
&\simeq 1 \qquad(|z-\xi|<\tau(\xi)).
\\
\label{BL2}
|F_{\xi}(z)|\,\omega(z)^{1/p}
&\lesssim \min \biggl(1,\frac{\min\big(\tau(\xi),\tau(z)\big)}{|z-\xi|} \biggr)^{3n} \quad(z\in \D).
\end{align}
Moreover, 
\begin{equation}\label{BL3}
\|F_{\xi}\|_{A^p_\omega}^p\simeq \tau(\xi)^2 \qquad (\rho_ 0\le|\xi|<1).
\end{equation}
\end{lettertheorem}

\begin{corollary}
Let $\omega\in\mathcal{W}$, $g\in\H(\D)$  and $0<p<\infty$.
Let $L_g$ be a $g$-operator written in the form  \eqref{eq:STrepre} with $P_{N+1}\neq 0$. Then  $L_g\in\BB(A^p_\omega)$ if and only if
$g\in H^\infty$.
\end{corollary}
\begin{proof}
Let $n\in\N$ with $np\ge 1$, and let us consider the functions $F_{\xi}$ of Theorem~\ref{le:Borichev} and 
\[
K_\xi=\frac{F_{\xi}}{\|F_{
\xi}\|_{A^p_\omega}} \qquad(\rho_0\le|\xi|<1).
\]
Then $\{K_\xi\}_{ \rho_0\le|\xi|<1}$ satisfies hypothesis \ref{item:th:main5:1} of Theorem~\ref{th:main5}.

On the other hand, if $|z|<r<1$ and $|\xi|\geq \max\{\frac{1+r}{2}, \rho_ 0\}$, then by \eqref{BL2}
\[ 
|F_{\xi}(z)|\lesssim \frac{1}{\omega(z)^\frac{1}{p}}\left( \frac{\tau(\xi)}{|z-\xi|} \right)^{3n}
\le  \frac{1}{\omega(r)^\frac{1}{p}} \left( \frac{2\tau(\xi)}{1-r} \right)^{3n}.
\]
Therefore, if $|z|<r<1$, by \eqref{BL3}
\[
|K_{\xi}(z)|\lesssim \frac{1}{\omega(r)^\frac{1}{p} (1-r)^{3n}} \tau(\xi)^{3n-2/p},
\]
then, bearing in mind that $3n-2/p>0$ and $\lim_{|\xi|\to 1^-}\tau(\xi)=0$, we get that
$\{K_\xi\}_{\rho_0\le|\xi|<1}$ fulfills  hypothesis \ref{item:th:main5:2} of Theorem~\ref{th:main5}.

Next, take $\zeta\in \T$ and $\delta>0$. If $|z-\zeta|\ge \delta$ and $|\xi-\zeta|<\frac{\delta}{2}$, then
$|z-\xi|\ge \frac{\delta}{2}$, so using  \eqref{BL2} and \eqref{BL3},
\[
|K_{\xi}(z)|^p\omega(z)\lesssim \frac{|F_{\xi}(z)|^p \omega(z)}{\tau(\xi)^2}
\lesssim \frac{\tau(\xi)^{3np-2}}{|z-\xi|^{3np}}\lesssim  \frac{\tau(\xi)^{3np-2}}{\delta^{3np}}.
\]
Therefore
\[
\int_{\D\setminus D(\zeta,\delta)}|K_\xi|^p\omega\,dA
\lesssim \frac{\tau(\xi)^{3np-2}}{\delta^{3np}},
\]
and hence $\{K_\xi\}_{\xi\in\D}$ satisfies \ref{item:th:main5:3}.
Consequently, an application of Theorem~\ref{th:main5} ends the proof.
\end{proof}

\vspace{1em}

{\textbf{Declaration of interest: none.}}

\end{document}